\documentclass[12pt]
{article}

\usepackage{showlabels}
\usepackage{amsmath}
\usepackage{amssymb}
\usepackage{fullpage}
\usepackage{amsthm}
\usepackage{graphicx}
\usepackage{placeins}
\usepackage{caption}
\usepackage{extpfeil}
\usepackage{enumerate}
\usepackage{caption}
\usepackage{subcaption}
\usepackage{hyperref}
\usepackage{xcolor}
\usepackage{pbox}

\newcommand{\be}{\begin{equation}}
\newcommand{\ee}{\end{equation}}
\newcommand{\bee}{\begin{equation*}}
\newcommand{\eee}{\end{equation*}}

\newcommand{\supp}{\text{\textup{supp}}\,}
\newcommand*{\avint}{\mathop{\ooalign{$\int$\cr$-$}}}

\newtheorem{thm}{Theorem}[section]
\newtheorem{prop}{Proposition}[section]

\newtheorem{rmk}{Remark}[section]
\newtheorem{lem}{Lemma}[section]

\newtheorem{cor}{Corollary}[section]\label{key}

\numberwithin{equation}{section}

\begin{document}

\title{Anisotropic Caffarelli-Kohn-Nirenberg type inequalities  }
\author{YanYan Li\footnote{Department of Mathematics, Rutgers University, 110 Frelinghuysen Road, Piscataway, NJ 08854, USA. Email: yyli@math.rutgers.edu,   is partially supported by 
NSF Grants DMS-1501004, DMS-2000261, and Simons Fellows Award 677077. }\ \  and  
Xukai Yan\footnote{Department of Mathematics,  Oklahoma State University, 401 Mathematical Sciences Building, Stillwater, OK 74078, USA. Email: xuyan@okstate.edu, is partially supported by AMS-Simons Travel Grant and AWM-NSF Travel Grant 1642548.}}
\date{}
\maketitle

\abstract{
\noindent Caffarelli, Kohn and Nirenberg considered in 1984 the interpolation inequalities  
\[
   \||x|^{\gamma_1}u\|_{L^s(\mathbb{R}^n)}\le C\||x|^{\gamma_2}\nabla u\|_{L^p(\mathbb{R}^n)}^a\||x|^{\gamma_3}u\|_{L^q(\mathbb{R}^n)}^{1-a} 
\] 
 in dimension $n\ge 1$,  and established necessary and sufficient conditions for which to hold under natural assumptions on the parameters.   Motivated by our study of the asymptotic stability of solutions to the Navier-Stokes equations, we consider a more general and improved anisotropic  version of the interpolation inequalities  
\[
    \||x|^{\gamma_1}|x'|^{\alpha}u\|_{L^s(\mathbb{R}^n)}\le C\||x|^{\gamma_2}|x'|^{\mu}\nabla u\|_{L^p(\mathbb{R}^n)}^{a}\||x|^{\gamma_3}|x'|^{\beta}u\|_{L^q(\mathbb{R}^n)}^{1-a} 
\] 
 in dimensions $n\ge 2$, where $x=(x', x_n)$ and $x'=(x_1, ..., x_{n-1})$,  and give necessary and sufficient conditions for which to hold  under 
 natural assumptions on the parameters.    Moreover we extend the Caffarelli-Kohn-Nirenberg inequalities from $q\ge 1$ to $q>0$. 
 This extension,  together with a nonlinear Poincar\'{e} inequality which we obtain in this paper, has played an important role 
 in our proof of the above mentioned anisotropic interpolation inequalities. }

\section{Introduction}

Let $n\ge 1$, $s,p,q,\alpha, \mu, \beta, a$ be real numbers satisfying  
\begin{equation}\label{eqNCB_1'}
  s>0, \quad p, q\ge 1, \quad 0\le a\le 1, 
\end{equation}
 \begin{equation}\label{eqNCB_2}
    \frac{1}{s}+\frac{\gamma_1}{n}>0, \quad \frac{1}{p}+\frac{\gamma_2}{n}>0, \quad \frac{1}{q}+\frac{\gamma_3}{n}>0, 
    \end{equation}
     \begin{equation}\label{eqNCB_3}
      \frac{1}{s}+\frac{\gamma_1}{n}=a\Big(\frac{1}{p}+\frac{\gamma_2-1}{n}\Big)+(1-a)\Big(\frac{1}{q}+\frac{\gamma_3}{n}\Big), 
    \end{equation}
     \begin{equation}\label{eqNCB_4}
     \gamma_1 \le a\gamma_2+(1-a)\gamma_3, 
    \end{equation}
     \begin{equation}\label{eqNCB_5}
     \frac{1}{s}\le \frac{a}{p}+\frac{1-a}{q} \quad \textrm{ if } a=0 \textrm{ or } a=1 \textrm{ or }\frac{1}{s}+\frac{\gamma_1}{n}=\frac{1}{p}+\frac{\gamma_2-1}{n}=\frac{1}{q}+\frac{\gamma_3}{n}. 
    \end{equation}

Caffarelli, Kohn and Nirenberg 
established the following classical interpolation inequalities. 

\medskip

\noindent\textbf{Theorem A. (\cite{CKN}, see also \cite{CKN2})}   \emph{For  $n\ge 1$, let 
   $s, p, q, \gamma_1, \gamma_2, \gamma_3$ and $a$ satisfy (\ref{eqNCB_1'}) and (\ref{eqNCB_2}). 
        Then there exists some positive constant $C$    such that 
               \begin{equation}\label{eqD_2_2_0}
     \||x|^{\gamma_1}u\|_{L^s(\mathbb{R}^n)}\le C\||x|^{\gamma_2}\nabla u\|_{L^p(\mathbb{R}^n)}^a\||x|^{\gamma_3}u\|_{L^q(\mathbb{R}^n)}^{1-a} 
  \end{equation}
  holds for all $u\in C^1_c(\mathbb{R}^n)$ if and only if  (\ref{eqNCB_3})-(\ref{eqNCB_5}) hold.   Furthermore, on any compact set in the parameter space in which (\ref{eqNCB_1'}) and (\ref{eqNCB_2}) hold, 
    the constant $C$ is bounded.}

\medskip

Given (\ref{eqNCB_1'}), condition (\ref{eqNCB_2}) holds if and only if $\||x|^{\gamma_1}u\|_{L^s(\mathbb{R}^n)}$, $\||x|^{\gamma_2}\nabla u\|_{L^p(\mathbb{R}^n)}$ and $\||x|^{\gamma_3}u\|_{L^q(\mathbb{R}^n)}$ are finite for all $u\in C_c^{\infty}(\mathbb{R}^n)$. 

The above theorem is the same as the theorem in \cite{CKN}, though the formulation of the conditions is somewhat different.

Lin \cite{Lin} generalized (\ref{eqD_2_2_0}) 
 to include derivatives of any order. 
 Badiale and Tarantello \cite{BT} derived a cylindrical Sobolev-Hardy type inequality.  
 Bahouri, Chemin and Gallagher \cite{BCG1, BCG2} obtained refined Hardy inequalities. 
 Nguyen and Squassina \cite{NS, NS2} generalized the Caffarelli-Kohn-Nirenberg inequalities to fractional Sobolev spaces. 
 Best constants and the existence (and nonexistence) of extremal functions of  (\ref{eqD_2_2_0}) 
have been studied extensively, see Catrina and Wang \cite{CW},  Dolbeault,  Esteban and Loss \cite{DEL}, and the references therein. 
 Partly motivated by works of 
  Bourgain,  Brezis  and  Mironescu \cite{BBM1, BBM2} and Maz’ya  and  Shaposhnikova  \cite{MS}, Frank and Seiringer \cite{FS} identified best constants for fractional Hardy type inequalities. 
  Sharp Sobolev and isoperimetric inequalities with monomial weights, and related problems, are studied by Cabr\'{e}, Ros-Oton and Serra, see \cite{CR, CRS}.

In this paper, we prove the following theorem on anisotropic Caffarelli-Kohn-Nirenberg type inequalities.

For $n\ge 2$, let  $s, p, q, a, \gamma_1, \gamma_2, \gamma_3, \alpha, \mu$ and $\beta$ be real numbers satisfying 

\begin{equation}\label{eqNCA_1}
    s, q>0, \quad p \ge 1, \quad 0\le a\le 1, 
    \end{equation}
    \begin{equation}\label{eqNCA_2}
    \frac{1}{s}+\frac{\alpha}{n-1}>0, \quad \frac{1}{p}+\frac{\mu}{n-1}>0,  
    \quad \frac{1}{q}+\frac{\beta}{n-1}>0,
    \end{equation}
    \begin{equation}\label{eqNCA_3}
      \frac{1}{s}+\frac{\alpha+\gamma_1}{n}>0,  \quad \frac{1}{p}+\frac{\mu+\gamma_2}{n}>0, \quad \frac{1}{q}+\frac{\beta+\gamma_3}{n}>0,
    \end{equation}
\begin{equation}\label{eqNCA_5}
      \frac{1}{s}+\frac{\gamma_1+\alpha}{n}=a\Big(\frac{1}{p}+\frac{\gamma_2+\mu-1}{n}\Big)+(1-a)\Big(\frac{1}{q}+\frac{\gamma_3+\beta}{n}\Big), 
      \end{equation}
      \begin{equation}\label{eqNCA_6_1}
   \gamma_1\le a\gamma_2+(1-a)\gamma_3,  
\end{equation}
 \begin{equation}\label{eqNCA_6_2}
  \gamma_1+\alpha\le a(\gamma_2+\mu)+(1-a)(\gamma_3+\beta), 
\end{equation}
 \begin{equation}\label{eqNCA_6_3}
  \frac{1}{s}+\frac{\alpha}{n-1}\ge a\Big(\frac{1}{p}+\frac{\mu-1}{n-1}\Big)+(1-a)\Big(\frac{1}{q}+\frac{\beta}{n-1}\Big), 
  \end{equation}
\begin{equation}\label{eqNCA_7}
\begin{split}
 \frac{1}{s}\le \frac{a}{p}+\frac{1-a}{q}  &  \textrm{ if } a=0 \textrm{ or } a=1 \textrm{ or }\frac{1}{p}+\frac{\gamma_2+\mu-1}{n}=\frac{1}{q}+\frac{\gamma_3+\beta}{n}=\frac{1}{s}+\frac{\gamma_1+\alpha}{n}, \\
& \textrm{ or }
  \frac{1}{s}+\frac{\alpha}{n-1}=a\Big(\frac{1}{p}+\frac{\mu-1}{n-1}\Big)+(1-a)\Big(\frac{1}{q}+\frac{\beta}{n-1}\Big).   \end{split}
\end{equation}

Throughout the paper, we denote $x=(x', x_n)$, where $x'=(x_1, ..., x_{n-1})$. We have the following theorem.
\begin{thm}\label{thm_main}
  For $n\ge 2$, let $s, p, q, a, \gamma_1, \gamma_2, \gamma_3, \alpha, \mu$ and $\beta$ be real numbers satisfying (\ref{eqNCA_1})-(\ref{eqNCA_3}). 
  Then there exists some positive constant $C$  such that 
   \begin{equation}\label{eqNC}
   \||x|^{\gamma_1}|x'|^{\alpha}u\|_{L^s(\mathbb{R}^n)}\le C\||x|^{\gamma_2}|x'|^{\mu}\nabla u\|_{L^p(\mathbb{R}^n)}^{a}\||x|^{\gamma_3}|x'|^{\beta}u\|_{L^q(\mathbb{R}^n)}^{1-a} 
\end{equation}
holds for all $u\in C_c^{1}(\mathbb{R}^n)$ if and only if 
  (\ref{eqNCA_5})-(\ref{eqNCA_7}) hold. 
  Furthermore, on any compact set in the parameter space in which (\ref{eqNCA_1})-(\ref{eqNCA_3}) hold, the constant $C$ is bounded. 
\end{thm}

Given (\ref{eqNCA_1}), conditions (\ref{eqNCA_2}) and (\ref{eqNCA_3}) hold if and only if $\||x|^{\gamma_1}|x'|^{\alpha}u\|_{L^s(\mathbb{R}^n)}$, \\
 $\||x|^{\gamma_2}|x'|^{\mu}\nabla u\|_{L^p(\mathbb{R}^n)}$ and $\||x|^{\gamma_3}|x'|^{\beta}u\|_{L^q(\mathbb{R}^n)}$
 are finite for all $u\in C_c^{\infty}(\mathbb{R}^{n})$. 

 Inequality (\ref{eqNC}) was proved in \cite{BT} in the special cases when $n\ge 3$, $a=1$, $\gamma_1=\gamma_2=\mu=0$, $1/s+\alpha/n=1/p-1/n$, $1<p<n$,  $-1\le \alpha\le 0$ and $(1/p-1/n)(n-1)+\alpha/n>0$;
 and proved in \cite{LY} in the special cases when $n\ge 2$, $a=1$, $1\le s=p<n$, $\alpha p>1-n$, 
 $\mu p>1-n$, $(\alpha+\gamma_1)p>-n$, 
 $\gamma_1+\alpha=\gamma_2+\mu-1$ and $\gamma_1\le \gamma_2$. 

Taking $\alpha=\mu=\beta=0$,  inequality (\ref{eqNC}) is an improvement 
of the Caffarelli-Kohn-Nirenberg inequalities 
(Theorem A)  
from $q\ge 1$ to $q>0$. 

When $\alpha<0$ and $\mu, \beta>0$, inequality (\ref{eqNC})  strengthens 
\[
     \||x|^{\gamma_1+\alpha}u\|_{L^s(\mathbb{R}^n)}\le C\||x|^{\gamma_2+\mu}\nabla u\|_{L^p(\mathbb{R}^n)}^a\||x|^{\gamma_3+\beta}u\|_{L^q(\mathbb{R}^n)}^{1-a}
\]
 which is given by (\ref{eqD_2_2_0}). 
In particular, when $n\ge 3$,  $a=1$, $s=p=2$, $\gamma_1=\gamma_2=\alpha=-1/2$, and $\mu=1/2$, 
inequality (\ref{eqNC}) takes the form
\begin{equation}\label{eqNSE_2}
   \int_{\mathbb{R}^n}\frac{|u|^2}{|x||x'|}dx\le C\int_{\mathbb{R}^n}|\nabla u|^2\frac{|x'|}{|x|}dx, 
\end{equation}
which strengthens the Hardy inequality
\[
   \int_{\mathbb{R}^n}\frac{|u|^2}{|x|^2}dx\le C\int_{\mathbb{R}^n}|\nabla u|^2dx. 
\]
Inequality (\ref{eqNSE_2}) was among the special cases proved in \cite{LY} as mentioned above. 

The necessity of (\ref{eqNCA_5}) is proved by dimensional analysis. The necessity of (\ref{eqNCA_6_1}) and (\ref{eqNCA_6_2}) are deduced from the fact that if (\ref{eqNC}) holds for $u$, then it also holds for 
$u(x)\mapsto u(x_1+S, x_2, ..., x_{n-1}, x_n+T)$  for all $S, T>0$. The necessity of (\ref{eqNCA_6_3}) and (\ref{eqNCA_7}) are more delicate. 
  
  A main ingredient of the proof of Theorem \ref{thm_main}, even in the case when $q\ge 1$, 
   is the above mentioned improvement of Theorem A from $q\ge 1$ to $q>0$, 
  which is stated as the following theorem.

 \begin{thm}\label{thmD_2}
   For  $n\ge 1$, 
    let $s, p, q, \gamma_1, \gamma_2, \gamma_3$,  and $a$ satisfy (\ref{eqNCA_1}) and (\ref{eqNCB_2}). 
        Then there exists some positive constant $C$,      such that 
               \begin{equation}\label{eqD_2_2}
     \||x|^{\gamma_1}u\|_{L^s(\mathbb{R}^n)}\le C\||x|^{\gamma_2}\nabla u\|_{L^p(\mathbb{R}^n)}^a\||x|^{\gamma_3}u\|_{L^q(\mathbb{R}^n)}^{1-a} 
  \end{equation}
  holds for all $u\in C^1_c(\mathbb{R}^n)$ if and only if  (\ref{eqNCB_3})-(\ref{eqNCB_5}) hold.   Furthermore, on any compact set in the parameter space in which (\ref{eqNCA_1}) and (\ref{eqNCB_2}) hold, 
    the constant $C$ is bounded. 
\end{thm}

Our proof of the sufficiency part of Theorem \ref{thmD_2}, which yields the extension from $q\ge 1$ to $q>0$,  is quite different from the proof of Theorem A in \cite{CKN}.

  Another main ingredient of the proof of Theorem \ref{thm_main} is the following nonlinear Poincar\'{e} inequality.

\begin{thm}[A nonlinear Poincar\'{e} inequality]\label{thm1-new}
For $n\ge 1$ and 
$0<\lambda<\infty$, assume $1\le p\le \infty$ if    $1\le \lambda<\infty$, and $\max\{1, n/(1+n\lambda)\}\le p\le \infty$ if $0<\lambda<1$. 
Let $(M, g)$ be an $n$-dimensional smooth compact Riemannian manifold without boundary,  $\Omega=M$ or 
$\Omega\subset M$ be an open connected set with Lipschitz boundary, and $S\subset \Omega$ has positive measure $|S|$. 
Then there exists some positive constant $C$,  depending only on $p, \lambda$, $\Omega$ and $S$, 
such that for every nonnegative $w\in W^{1,p}(\Omega)$,
\begin{equation}\label{est1-new}
\|w-\big(\avint_Sw^{1/\lambda}\big)^{\lambda}\|_{ L^p(\Omega) }\le C \|\nabla w \|_{ L^p(\Omega) },
\end{equation}
On the other hand, if $0<\lambda<1$ and $0<p<n/(1+n\lambda)$ or $0<\lambda<\infty$ and $0<p<1$, there does not exist $C$ for which
 (\ref{est1-new}) holds.
\end{thm}

Theorem \ref{thm1-new} gives necessary and sufficient conditions on $(\lambda, p)\in (0, \infty)\times (0, \infty]$ for (\ref{est1-new}) to hold.

\begin{cor}\label{cor_new}
  For  $n\ge 1$,   $1\le p\le \infty$, and $0<q<\infty$,   let $(M, g)$ be an $n$-dimensional smooth compact Riemannian manifold without boundary,  $\Omega\subset M$ be an open connected set with Lipschitz boundary, and $S\subset \Omega$ has positive measure $|S|$. 
Then there exists some positive constant $C$,  depending only on $p, q, \Omega$ and $S$, 
such that for every nonnegative $w\in W^{1,p}(\Omega)$,
\begin{equation}\label{est1-newcor}
\|w\|_{L^p(\Omega)}\le \big(\avint_Sw^{q}\big)^{1/q}|\Omega|^{1/q}+ 
C \|\nabla w \|_{ L^p(\Omega) }. 
\end{equation}
\end{cor}

 \bigskip
 
 
 Theorem \ref{thm_main} grows out of 
  our study in \cite{LY} on the stability of 
  solutions to
  Navier-Stokes equations. 
In  joint work with L. Li  \cite{LLY1, LLY}, we classified $(-1)$-homogeneous axisymmetric no-swirl solutions of the $3$D incompressible stationary Navier-Stokes equations which are smooth in $\mathbb{R}^3$ away from the symmetry axis. All such solutions $u$ are of 
the following three mutually exclusive types:
 \begin{enumerate}
  \item[]Type 1.   Landau solutions, which satisfy $\displaystyle \sup_{|x|=1}|\nabla u(x)|<\infty$;
   \item[]Type 2. Solutions satisfying $\displaystyle 0<\limsup_{|x|=1,x'\to 0}|x'||\nabla u(x)|<\infty$; 
   \item[] Type 3. Solutions satisfying $\displaystyle \limsup_{|x|=1,x'\to 0}|x'|^2|\nabla u(x)|>0$.
 \end{enumerate}
 Karch and Pilarczyk \cite{Karch} proved the asymptotic stability of Landau solutions under $L^2$-perturbations. 
 In \cite{LY}, 
 we proved  the asymptotic stability of  Type 2 solutions under $L^2$-perturbations. 
  An important ingredient in our proof is the following improved version of Hardy's inequality 
\begin{equation}\label{eqNSE_1}
  \int \frac{|u|^2}{|x||x'|}\le C\int |\nabla u|^2, 
\end{equation}
a weaker form of (\ref{eqNSE_2}).  
We expect that Theorem \ref{thm_main} will be useful in the study of the asymptotic stability of Type 3 solutions and the stability of Type 2 solutions in other function spaces. 
For related results on the stability of singular solutions  
 to the Navier-Stokes equations, see \cite{CKPW}, \cite{KPS}, \cite{LZZ}, \cite{ZZ} and the reference therein. 
 
 In the special case when $a=1$ and $ 1\le s=p<n$, Theorem \ref{thm_main} says that 
   under the conditions 
  \begin{equation}\label{eqB_2_0}
    p\ge 1, \quad \frac{1}{p}+\frac{\mu}{n-1}>0,\quad  \frac{1}{p}+\frac{\gamma_1+\alpha}{n}>0,
 \end{equation}
 \begin{equation}\label{eqB_3_0}
    \gamma_1+\alpha=\gamma_2+\mu-1, \quad  \gamma_1\le \gamma_2,
    \end{equation}
   inequality 
   \begin{equation}\label{eqB_1_0}
   \||x|^{\gamma_1}|x'|^{\alpha}u\|_{L^p(\mathbb{R}^n)}\le C\||x|^{\gamma_2}|x'|^{\mu}\nabla u\|_{L^p(\mathbb{R}^n)}
 \end{equation}
 holds for $u\in C_c^1(\mathbb{R}^n)$. 
    This was proved in \cite{LY} when $1\le p<n$  as mentioned earlier. The proof there applies to $p\ge n$ as well, and we present the proof concisely below.

    The necessity of  $\gamma_1+\alpha=\gamma_2+\mu-1$  follows from a  dimensional analysis argument.  The necessity of $\gamma_1\le \gamma_2$ can also be seen easily: take a unit ball $B_i$ centered at $x_i=(x_i', i)$ with $2<|x_i'|<3$ and let $u_i(\cdot)=v(\cdot+x_i)$.

Let $x=(r, \theta)$ in spherical coordinates. 
Since $(\gamma_1+\alpha)p+n-1>-1$, we have, for each fixed $\theta$, that 
\begin{equation}
   \int_{0}^{\infty}r^{(\gamma_1+\alpha)p+n-1}|u|^pdr\le C\int_{0}^{\infty}r^{(\gamma_1+\alpha+1)p+n-1}|\partial_ru|^pdr. 
  \end{equation}
For any $0<\epsilon<\pi/4$, let $K_{\epsilon}:=\{x\in\mathbb{R}^n\mid |x'|\le |x|\sin\epsilon \}$. Integrate the above  in $\theta$ 
on $(\mathbb{R}^{n}\setminus K_{\epsilon})\cap\mathbb{S}^{n-1}$, we have, using $ |x|\sin\epsilon \le |x'|\le |x|$ in $\mathbb{R}^n\setminus K_{\epsilon}$ and $\gamma_1+\alpha+1=\gamma_2+\mu$, that 
 \begin{equation}\label{eqB_1}
    \||x|^{\gamma_1}|x'|^{\alpha}u\|_{L^p(\mathbb{R}^n\setminus K_{\epsilon})}\le C\||x|^{\gamma_2}|x'|^{\mu}\partial_ru\|_{L^p(\mathbb{R}^n\setminus K_{\epsilon})}.
 \end{equation}
 Since $\int_{K_{2\epsilon}\setminus K_{\epsilon}}|x|^{\gamma_1 p}|x'|^{\alpha p}|u|^pdx=\int_{\epsilon}^{2\epsilon}\int_{\partial K_{\delta}}|x|^{\gamma_1 p}|x'|^{\alpha p}|\nabla u|^p\frac{|x|}{\cos\delta}d\sigma(x)d\delta$,  there is some 
 $\epsilon<\delta<2\epsilon$, such that 
\begin{equation}\label{eqB_2}
\||x|^{\gamma_1}|x'|^{\alpha}|x|^{1/p}u\|_{L^p(\partial K_{\delta})}
\le C\||x|^{\gamma_2}|x'|^{\mu} \partial_r u\|_{L^p(K_{2\epsilon}\setminus K_{\epsilon})}.
\end{equation}
Next, let $x=(r, \theta)=(r, \theta_1, \omega)$,  where $r=|x|$, $\omega\in\mathbb{S}^{n-2}$ and $\theta_1$ is the polar angle, i.e. the angle between the $x_n$-axis and the ray from the origin to $x$. Then $|x'|=|x|\sin\theta_1$, and  
\[
   |u(r, \theta_1,\omega)|^p- |u(r, \delta, \omega)|^p=-\int_{\theta_1}^{\delta}\partial_t|u(r, t, \omega)|^pdt\le \int_{\theta_1}^{\delta }p|u|^{p-1}|\partial_t u|dt.
\]
We multiply the above by $\theta_1^{\alpha p+n-2}$ and integrate in $\theta_1$ over $[0, \delta]$. 
We know from (\ref{eqB_2_0}) and (\ref{eqB_3_0}) that $\alpha p+n-1>0$. So we have 
\[
   \int_{0}^{\delta}\theta_1^{\alpha p+n-2}\int_{\theta_1}^{\delta }p|u|^{p-1}|\partial_t u|dt d\theta_1
  \le C\int_{0}^{\delta}\theta_1^{\alpha p+n-1}|u|^{p-1}|\partial_{\theta_1} u|d\theta_1,
\]
and 
\[
    \int_{0}^{\delta} \theta_1^{\alpha p+n-2}|u(r, \theta_1,\omega)|^pd\theta_1\le C|u(r, \delta,\omega)|^p+C\int_{0}^{\delta}\theta_1^{\alpha p+n-1}|u|^{p-1}|\partial_{\theta_1} u|d\theta_1.
\]
Using Cauchy-Schwarz inequality, we have
\[
   \int_{0}^{\delta} \theta_1^{\alpha p+n-2}|u(r, \theta_1,\omega)|^pd\theta_1\le C|u(r, \delta,\omega)|^p+C\int_{0}^{\delta}\theta_1^{(\alpha+1)p+n-2}|\partial_{\theta_1}u(r, \theta_1,\omega)|^pd\theta_1.
\]
Multiplying the above by $r^{(\gamma_1+\alpha)p+n-1}$ and integrating in $r$ and $\omega$, we have, using the fact that $\gamma_1+\alpha+1=\gamma_2+\mu$,  $\theta_1^{\alpha+1}\le \theta_1^{\mu}$ on $[0, \delta]$ in view of $\mu\le \alpha+1$, and $|\partial_{\theta_1}u|/r\le |\nabla u|$, that 
\[
   \||x|^{\gamma_1}|x'|^{\alpha}u\|_{L^p(K_{\delta})}\le C\||x|^{\gamma_1}|x'|^{\alpha}|x|^{1/p}u\|_{L^p(K_{\delta})}+C\||x|^{\gamma_2}|x'|^{\mu}\nabla u\|_{L^p(K_{\delta})}. 
\] 
Inequality (\ref{eqB_1_0}) follows from (\ref{eqB_1}), (\ref{eqB_2}) and the above. 
 
 For  the  sufficiency part of Theorem \ref{thm_main} when $s\ne p$ or $0<a<1$, the proof is more involved. 
 Let us look at  a simple case  where $n=2$, $a=1$, $\gamma_1=\gamma_2=\gamma_3=0$, $s=2$, $p=1$,  $\mu=\alpha>-1/2$.   The following lemma is a slightly stronger version of Theorem \ref{thm_main} in this case. 
  
 \begin{prop}\label{lemS_1}
   For $\alpha>-1/2$, there exists some constant $C$ depending only on $\alpha$ such that    \begin{equation}\label{eqS_1_1}
       \||x_1|^{\alpha}u\|^2_{L^2([0, 1]^2)}\le C\||x_1|^{\alpha}\partial_{x_1}u\|_{L^1([0, 1]^2)}\||x_1|^{\alpha}\partial_{x_2}u\|_{L^1([0, 1]^2)}
   \end{equation}
   holds for all $u\in C^1([0, 1]^2)$ satisfying $u(1, x_2)=u(x_1, 1)=0$, $0\le x_1, x_2\le 1$.  \end{prop}
 \begin{proof}
   Make a change of variables $y_1=x_1^{2\alpha+1}$, $y_2=x_2$, and $\tilde{u}(y_1, y_2)=u(x_1, x_2)$.     Then  (\ref{eqS_1_1}) is equivalent to 
    \begin{equation}\label{eqS_1_2}
       \|\tilde{u}\|^2_{L^2([0, 1]^2)}\le C\||y_1|^{\beta}\partial_{y_1}\tilde{u}\|_{L^1([0, 1]^2)}\||y_1|^{-\beta}\partial_{y_2}\tilde{u}\|_{L^1([0, 1]^2)}, 
           \end{equation}
            for  $\beta<1/2$ (where $\beta:=\alpha/(2\alpha+1)$)  and 
           $\tilde{u}\in C^1([0, 1]^2)$ satisfying $\tilde{u}(y_1, 1)=\tilde{u}(1, y_2)=0$, $0\le y_1, y_2\le 1$. 
           
           For $k\in \mathbb{N}$, let $R_k=[2^{-k-1}, 2^{-k}]$, 
           \[
              A_k:=\|\tilde{u}\|_{L^2(R_k\times [0, 1])}, \quad P_k:=\|y_1^{\beta}\partial_{y_1}\tilde{u}\|_{L^1(R_k\times [0, 1])}, \quad Q_k:=\|y_1^{-\beta}\partial_{y_2}\tilde{u}\|_{L^1(R_k\times [0, 1])},            \]
           and 
           \[
              P:=\||y_1|^{\beta}\partial_{y_1}\tilde{u}\|_{L^1([0, 1]^2)}, \quad Q:=\||y_1|^{-\beta}\partial_{y_2}\tilde{u}\|_{L^1([0, 1]^2)}. 
           \]
           For any $y_1\in R_k$, $\xi\in R_{k-1}$ and $y_2\in [0, 1]$, we have 
                      \[
             \begin{split}
              |\tilde{u}(y_1, y_2)|^2 & =|\tilde{u}(y_1, y_2)||\tilde{u}(\xi, y_2)-\int_{y_1}^{\xi}\partial_{\eta}\tilde{u}(\eta, y_2)d\eta|\\
              & \le  |\tilde{u}(y_1, y_2)||\tilde{u}(\xi, y_2)|+Cy_1^{-\beta}|\tilde{u}(y_1, y_2)|\int_{y_1}^{\xi}\eta^{\beta}|\partial_{\eta}\tilde{u}(\eta, y_2)|d\eta\\
              & \le |\tilde{u}(y_1, y_2)||\tilde{u}(\xi, y_2)|+C\int_{0}^{1}y_1^{-\beta}|\partial_{y_2}\tilde{u}(y_1, y_2)|dy_2 \int_{R_k\cup R_{k-1}}\eta^{\beta}|\partial_{\eta}\tilde{u}(\eta, y_2)|d\eta. 
              \end{split}
           \]
           Taking $\int_{0}^{1}\avint_{R_{k-1}}\int_{R_k}\cdot dy_1d\xi dy_2$ of the above and using H\"{o}lder's inequality, we have 
           \[
             \begin{split}
              A_k & \le \frac{1}{|R_{k-1}|}\int_{0}^{1} \Big(\int_{R_k}|\tilde{u}(y_1, y_2)|dy_1\Big)\Big( \int_{R_{k-1}}|\tilde{u}(\xi, y_2)|d\xi\Big)dy_2+CQ_kP\\ 
                            & \le 
              \sqrt{\frac{|R_k|}{|R_{k-1}|}}\sqrt{A_kA_{k-1}}+CQ_kP  
                             \le \frac{1}{2\sqrt{2}}(A_k+A_{k-1})+CQ_kP. 
              \end{split}
           \]
           Thus 
           \[
              A_k\le \theta A_{k-1}+CQ_kP,  
           \]
           where $\theta=1/(2\sqrt{2}-1)<1$. 
           Suming over $k\ge 1$ gives  
                      \begin{equation}\label{eqS_1_3}
               \sum_{k=1}^{\infty}A_k\le 
               \frac{1}{1-\theta} A_0+CPQ\le CPQ. 
           \end{equation}
           where we have used $|\tilde{u}(y_1, y_2)|\le \int_{0}^1|\partial_{y_1}\tilde{u}(y_1, \cdot)|$ and  $|\tilde{u}(y_1, y_2)|\le \int_{0}^1|\partial_{y_2}\tilde{u}(\cdot, y_2)|$. 
           \end{proof}

 \medskip

 For the sufficiency part of Theorem \ref{thm_main} in general,  our first consideration was for $q\ge 1$. We were able to prove the sufficiency part of Theorem \ref{thm_main} for $q\ge 1$ in dimension $n$ provided Theorem A  for $q>0$ in dimension $n-1$, with the help of the nonlinear Poincar\'{e}'s inequality (Theorem \ref{thm1-new}).  We also proved  Theorem A for $q>0$ in dimension $n=1$ and therefore proved Theorem \ref{thm_main} for $q\ge 1$ in dimension $n=2$ as well as Theorem \ref{thm_main} 
 for axisymmetric $u$ and $q\ge 1$ in dimensions $n\ge 3$.

Next we established the sufficiency part of Theorem \ref{thm_main} for $q\ge 1$ in dimensions $n\ge 3$. A key step is to prove (\ref{eqNC}) on a cylinder $D:=\{x\in \mathbb{R}^n\mid |x'|\le 1, 0\le x_n\le 1\}$ when $\gamma_1=\gamma_2=\gamma_3=0$. 
  For simplicity, one may consider $u\in C^1_c(D)$ and the estimate is 
  \begin{equation}\label{eqS_2_0}
      \||x'|^{\alpha}u\|_{L^s(D)}\le C\||x'|^{\mu}\nabla u\|_{L^p(D)}^a\||x'|^{\beta}u\|_{L^{q}(D)}^{1-a}. 
  \end{equation} 
  The left hand side of the above can be written as 
  \[
 \begin{split}
   \||x'|^{\alpha}u\|_{L^s(D)}
 &  =\||x'|^{\frac{\alpha s}{\bar{s}}}|u|^{\frac{s}{\bar{s}}}\|_{L^{\bar{s}}(D)}^{\frac{\bar{s}}{s}}
  \le C\||x'|^{\frac{\alpha s}{\bar{s}}}(|u|^{\frac{s}{\bar{s}}}-|u^*|^{\frac{s}{\bar{s}}})\|_{L^{\bar{s}}(D)}^{\frac{\bar{s}}{s}}+ C\||x'|^{\frac{\alpha s}{\bar{s}}}|u^*|^{\frac{s}{\bar{s}}}\|_{L^{\bar{s}}(D)}^{\frac{\bar{s}}{s}}\\
&  =:C(I_1+I_2), 
 \end{split}
 \]
 where $1/\bar{s}=1/s+1-1/p$ and $u^*(x',x_n)=\avint_{|y'|=|x'|}u(y',x_n) d\sigma (y')$. 
 Since 
 Theorem \ref{thm_main} for $q\ge 1$
   holds for axisymmetric $u$, so does (\ref{eqS_2_0}). Thus  $I_2$ is  bounded by  the right hand side of (\ref{eqS_2_0}).  
The estimate that $I_1$ is bounded by the right hand side of (\ref{eqS_2_0}) follows from  a variant of the Caffarelli-Kohn-Nirenberg inequalities (see Theorem \ref{thm6_1}), using the fact that $0<\bar{s}\le s$.

 Later we proved Theorem A for $q>0$ in all dimensions and in turn proved Theorem \ref{thm_main}. This is the proof presented in this paper.

  \medskip
 
 In Section \ref{sec_2}, we prove the necessity part of Theorem \ref{thm_main} and \ref{thmD_2}. In Section \ref{sec_3}, we prove Theorem \ref{thm1-new} and Corollary \ref{cor_new}. In Section \ref{sec_4}, we prove the sufficiency part of Theorem \ref{thmD_2} by establishing Theorem 4.1, a more general result including  inequalities on cones. 
 In Section \ref{sec_5}, we prove the sufficiency part of Theorem \ref{thm_main}. In Section \ref{sec_6}, we give two variants of Theorem A and Theorem \ref{thm_main}. Some properties of the parameters used in the proofs are given in the appendix.

\section{Proof of the necessity parts of Theorem \ref{thm_main} and  \ref{thmD_2}}
\label{sec_2}

In this section, we prove the necessity parts of Theorem \ref{thm_main} and  \ref{thmD_2}. 
We first prove the necessity part of Theorem \ref{thm_main} by the following lemma.  
\begin{lem}\label{lemNC_1}
   For $n\ge 2$, 
  let $s, p, q, a, \gamma_1, \gamma_2, \gamma_3, \alpha, \mu$ and $\beta$ satisfy (\ref{eqNCA_1})-(\ref{eqNCA_3}). 
   If there exists a constant $C$ such that (\ref{eqNC}) holds for all  $u$ in $C^{\infty}_c(\mathbb{R}^n)$, then (\ref{eqNCA_5})-(\ref{eqNCA_7}) hold.
\end{lem}
\begin{proof}
Let $C$ denote a positive constant depending only on $s, p, q, a, \gamma_1, \gamma_2, \gamma_3, \alpha, \mu$ and $\beta$  which may vary from line to line. 
 We prove (\ref{eqNCA_5})-(\ref{eqNCA_7}) one by one. 

\medskip

\noindent\emph{Proof of (\ref{eqNCA_5})}:
Fixing a $v\in C^{\infty}_c(\mathbb{R}^n)\setminus\{0\}$ and plugging $u(x):=v(\lambda x)$, $\lambda>0$, into (\ref{eqNC}), we have
 \[
     \lambda^{-nA_1} \||x|^{\gamma_1}|x'|^{\alpha}v\|_{L^s(\mathbb{R}^n)}\le C\lambda^{-nA_2}\||x|^{\gamma_2}|x'|^{\mu}\nabla v\|_{L^p(\mathbb{R}^n)}^{a}\||x|^{\gamma_3}|x'|^{\beta}v\|_{L^q(\mathbb{R}^n)}^{1-a}, 
   \]
   where $A_1$ and $A_2$ are the left and right hand side of  (\ref{eqNCA_5}) respectively. 
   Sending $\lambda$ to $0$ and $\infty$ in the above, we obtain (\ref{eqNCA_5}).

   \medskip
   
   \noindent\emph{Proof of (\ref{eqNCA_6_1}) and (\ref{eqNCA_6_2})}:   
  Fixing a $v\in C_c^{\infty}(B_1(0))\setminus\{0\}$,    
   we consider $u(x):=v(x-x_0)$ where $x_0=(S, 0,...,0, R)$ and $S,R>0$. Then $u\in C^{\infty}_c(B_1(x_0))$, and $u$ satisfies (\ref{eqNC}). 
   Choose $S=2$ and $R$ large. 
    For $x\in B_1(x_0)$, we have $2\le |x'|\le 3$ and $R/2\le |x|\le 2R$. 
    Plugging $u$ into (\ref{eqNC}),  we have, 
   \[
      R^{\gamma_1}\le CR^{a\gamma_2+(1-a)\gamma_3} 
   \]
   for some constant $C$ independent of $R$. Inequality (\ref{eqNCA_6_1})  follows, since $R$ can be arbitrarily large.

   Now we choose large $S$ and $R=0$.   
   For  $x\in B_1(x_0)$, both $|x'|$ and $|x|$ are in $[S/2, S]$. 
    Plugging $u$ into (\ref{eqNC}), we have 
   \[
      S^{\gamma_1+\alpha}\le CS^{a(\gamma_2+\mu)+(1-a)(\gamma_3+\beta)}.
   \]
Inequality (\ref{eqNCA_6_2})  follows from the above, since $S$ can be arbitrarily large.

\bigskip

Next, to prove (\ref{eqNCA_6_3}) and (\ref{eqNCA_7}), we fix a $g\in C_c^{\infty}(1, 4)$ satisfying 
\begin{equation*}
      g(t)=\left\{
         \begin{split}
            & 0, \quad t\le 1 \textrm{ or }t\ge 4, \\
            &  1, \quad 2\le t\le 3.
         \end{split}
      \right.
   \end{equation*} 
   \noindent\emph{Proof of (\ref{eqNCA_6_3})}: 
    For $0<\epsilon<1$, let 
   
    \[
      f_1(\rho)=\left\{
         \begin{array}{ll}
             0, & \rho \ge 2\epsilon, \\
             2\epsilon-\rho, & \epsilon\le \rho\le 2\epsilon, \\
                           \epsilon, & \rho\le \epsilon.  
                                    \end{array}
      \right.
   \]
     Then $u(x): =f_1(|x'|)g(x_n)$ satisfies (\ref{eqNC}). 
     We have $\mathrm{supp}$ $u\subset \{|x'|\le 2\epsilon, 1\le x_n\le 4\}$. For any $x$ in $\mathrm{supp}$ $u$, $1\le |x|\le 5$. Then (\ref{eqNC}) for this $u$ is equivalent to 
   \begin{equation*}
      \||x'|^{\alpha}u\|_{L^s(\mathbb{R}^n)}\le C\||x'|^{\mu}\nabla u\|_{L^p(\mathbb{R}^n)}^{a}\||x'|^{\beta}u\|_{L^q(\mathbb{R}^n)}^{1-a} 
   \end{equation*}
   for some constant $C$ independent of $\epsilon$. 
   By calculation,   
   \begin{equation*}
         \int_{\mathbb{R}^n}||x'|^{\alpha}u|^s dx \ge \frac{1}{C}\int_{5\epsilon/4\le |x'|\le 7\epsilon/4}||x'|^{\alpha}f_1(|x'|)|^sdx'\ge \frac{1}{C}\epsilon^{(\alpha+1)s+n-1},
   \end{equation*}

  \begin{equation*}
      \int_{\mathbb{R}^{n}}||x'|^{\mu}\nabla u|^pdx   \le C\int_{|x'|\le 2\epsilon}|x'|^{p\mu}(| f'_1(|x'|)|^p+|f_1(|x'|)|^p)dx'
                                                                               \le C\epsilon^{\mu p+n-1}, 
                                                                                  \end{equation*}
 \begin{equation*}
      \int_{\mathbb{R}^{n}}||x'|^{\beta} u|^qdx  \le C\int_{|x'|\le 2\epsilon}||x'|^{\beta}f_1(|x'|)|^qdx' 
                                                                                 \le C\epsilon^{(\beta+1) q+n-1}.    \end{equation*}
   Thus we have 
   \[
      \epsilon^{\alpha+1+(n-1)/s}\le C\epsilon^{a(\mu +(n-1)/p)+(1-a)(\beta +1+(n-1)/q)}. 
   \]
   Inequality (\ref{eqNCA_6_3}) follows, since $\epsilon$ can be arbitrarily small. 
  
  \bigskip
  
   \noindent\emph{Proof of (\ref{eqNCA_7})}:
  We divide the proof into two cases. 
  
   \medskip
    
 \noindent \textbf{Case 1.} $a=0$ or $a=1$ or $1/p+(\gamma_2+\mu-1)/n=1/q+(\gamma_3+\beta)/n=1/s+(\gamma_1+\alpha)/n$.
 
 \medskip
  
  We first prove the inequality in (\ref{eqNCA_7}) when $1/p+(\gamma_2+\mu-1)/n=1/q+(\gamma_3+\beta)/n=1/s+(\gamma_1+\alpha)/n$. 
  For $0<\epsilon<1$,  let 
   \begin{equation}\label{eqNC_1_f2}
     \displaystyle  f_2(r)=\left\{
         \begin{array}{ll}
             \displaystyle  r^{-\alpha-\gamma_1-n/s+\epsilon}, & 0< r\le  1, \\
            \displaystyle  1, & 1\le r\le 2, \\
            \displaystyle  \frac{4-r}{2}, & 2\le r\le 4,\\
            \displaystyle  0, & r\ge 4. 
         \end{array}
      \right.
   \end{equation}
    Then $u(x):=f_2(|x|)g(|x_n|/|x'|)$ satisfies (\ref{eqNC}) by the approximation of 
   $u_{\delta}(x):=f_2(\sqrt{|x|^2+\delta^2})g(|x_n|/|x'|)$. 
  By computation, 
   \begin{equation}\label{eqNC_1_5}
         \int_{\mathbb{R}^n}||x|^{\gamma_1}|x'|^{\alpha}u|^s dx  
         \ge  \frac{1}{C}\int_{0<|x|\le 1, 2\le |x_n|/|x'|\le 3}||x|^{\gamma_1+\alpha}f_2|^sdx
                   \ge \frac{1}{C}\int_{0}^{1}r^{\epsilon s-1}dr
          \ge \frac{1}{C}\epsilon^{-1}. 
    \end{equation}
Notice that for $1\le |x_n|/|x'|\le 4$, $|\nabla g(|x_n|/|x'|)|\le C|g'(|x_n|/|x'|)|/|x'|\le C|x|^{-1}$, we have, using the fact that $1/s+(\gamma_1+\alpha)/n=1/p+(\gamma_2+\mu-1)/n>0$, that 
  \begin{equation}\label{eqNC_1_6}
     \begin{split}
      \int_{\mathbb{R}^{n}}||x|^{\gamma_2}|x'|^{\mu}\nabla u|^pdx  & \le C\int_{|x|\le 1, 1\le |x_n|/|x'|\le 4}|x|^{p(\gamma_2+\mu)}\big(|\nabla f_2|^p+|x|^{-p}|f_2|^p\big)dx\\
                                                                              & \le C \int_{0}^{4} r^{(\gamma_2+\mu)p+n-1+(-\alpha-\gamma_1-n/s-1+\epsilon)p}dr\\                                                                               &= C\int_{0}^{4}r^{\epsilon p-1}dr
                                                                               \le C\epsilon^{-1}. 
                                                                                     \end{split}
   \end{equation}
Similarly, using the fact $1/s+(\gamma_1+\alpha)/n=1/q+(\gamma_3+\beta)/n>0$, we have  
 \begin{equation}\label{eqNC_1_7}
      \int_{\mathbb{R}^{n}}||x|^{\gamma_3}|x'|^{\beta} u|^qdx  \le C\int_{|x|\le 1, 1\le |x_n|/|x'|\le 4}||x|^{\gamma_3+\beta}f_2(x)|^qdx   \le C\epsilon^{-1}. 
         \end{equation}
   By (\ref{eqNC}), (\ref{eqNC_1_5}), (\ref{eqNC_1_6}) and (\ref{eqNC_1_7}), we have 
   \[
   \epsilon^{-1/s}\le C\epsilon^{-a/p-(1-a)/q} 
   \]
   for arbitrarily small $\epsilon$.  So the inequality in (\ref{eqNCA_7}) holds. 
   
  Now we turn to $a=0$ or $a=1$.  In view of (\ref{eqNCA_5}),   when $a=0$,  we have $1/s+(\gamma_1+\alpha)/n=1/q+(\gamma_3+\beta)/n$, and when $a=1$, we have  $1/s+(\gamma_1+\alpha)/n=1/p+(\gamma_2+\mu-1)/n$. The inequality in (\ref{eqNCA_7}) follows from the same proof as above. 
  
    \medskip
    
\noindent\textbf{Case 2.} $0<a<1$, $1/p+(\gamma_2+\mu-1)/n\ne  1/q+(\gamma_3+\beta)/n$,  and 
   \begin{equation}\label{eqNC_1_11}
      \frac{1}{s}+\frac{\alpha}{n-1}=a\Big(\frac{1}{p}+\frac{\mu-1}{n-1}\Big)+(1-a)\Big(\frac{1}{q}+\frac{\beta}{n-1}\Big).
   \end{equation}
    If (\ref{eqNC_1_11}) holds, then either Case 1 or Case 2 holds. 
  
   \medskip
   
  We divide the proof of Case 2 into two subcases. 
   \medskip
  
  \noindent \emph{Subcase 2.1.} $1/p+(\mu-1)/(n-1)=1/q+\beta/(n-1)$. 
  
   \medskip
  
  In this subcase, we have, in view of (\ref{eqNC_1_11}), that $1/s+\alpha/(n-1)=1/p+(\mu-1)/(n-1)=1/q+\beta/(n-1)$.  
  For $0<\epsilon<1$, let 
  \[
      f_3(\rho)=\left\{
         \begin{array}{ll}
            \displaystyle \rho^{-\alpha-(n-1)/s+\epsilon}, & 0< \rho\le  1, \\
                        \displaystyle 1, & 1\le \rho\le 2, \\
             \displaystyle \frac{4-\rho}{2}, & 2\le \rho\le 4,\\
            \displaystyle 0, & \rho\ge 4. 
         \end{array}
      \right.
   \]
Let $u(x):=f_3(|x'|)g(x_n)$. Then it satisfies (\ref{eqNC}) by the approximation of 
   $u_{\delta}(x): =f_3(\sqrt{|x'|^2+\delta^2})g(x_n)$. 
     By computation, we have
   \begin{equation}\label{eqNC_1_8}
         \int_{\mathbb{R}^n}||x|^{\gamma_1}|x'|^{\alpha}u|^s dx   \ge \frac{1}{C}\int_{0\le |x'|\le 1}||x'|^{\alpha}f_3(|x'|)|^sdx'
                                                                                                 \ge \frac{1}{C}\int_{0}^{1}\rho^{-1+\epsilon s}d\rho 
                                                                                                 \ge \frac{1}{C}\epsilon^{-1}.    \end{equation}
Since  $1/s+\alpha/(n-1)=1/p+(\mu-1)/(n-1)>0$, we have
  \begin{equation}\label{eqNC_1_9}
     \begin{split}
      \int_{\mathbb{R}^{n}}||x|^{\gamma_2}|x'|^{\mu}\nabla u|^pdx  & \le C\int_{|x'|\le 1}|x'|^{\mu p}(|\nabla f|^p+|f|^p)dx'\\ 
                                                                                    & \le C\int_{0}^{4}\rho^{\mu p+(\alpha-(n-1)/s-1+\epsilon)p+n-2}d\rho\\ 
                                                                                                                                                                  & = C\int_{0}^{4}\rho^{\epsilon p-1}d\rho 
                                                                               \le C\epsilon^{-1}. 
                                                                                     \end{split}
   \end{equation}
Similarly,  since $1/s+\alpha/(n-1)=1/q+\beta/(n-1)>0$, we have  
 \begin{equation}\label{eqNC_1_10}
      \int_{\mathbb{R}^{n}}||x|^{\gamma_3}|x'|^{\beta} u|^qdx  \le C\int_{|x'|\le 1}||x'|^{\beta}f_3(x')|^qdx' \le C\epsilon^{-1}.
         \end{equation}
   So by (\ref{eqNC}), (\ref{eqNC_1_8}), (\ref{eqNC_1_9}) and (\ref{eqNC_1_10}), we have 
   \[
     \epsilon^{-1/s}\le C\epsilon^{-a/p-(1-a)/q} 
   \]
   for arbitrarily small $\epsilon$. So the inequality in (\ref{eqNCA_7}) follows in this subcase.    
    \medskip
   
  \noindent \emph{Subcase 2.2.}  $1/p+(\mu-1)/(n-1)\ne 1/q+\beta/(n-1)$. 
  
   \medskip

    Introduce the spherical coordinates in $\mathbb{R}^n_+$: $r=|x|$,  $\theta=x/|x|$. 
   Let $\theta'=x'/|x|$, we have $\theta=(\theta', \sqrt{1-|\theta'|^2})$. For simplicity, we denote $x=(r, \theta')$. 
  Fix $\delta>0$ small, and  let $R_0:=\{x\in \mathbb{R}^n_+ \mid 1<r<2, \delta<|\theta'|<2\delta\}$. 
   Fix a function $u\in C_c^{\infty}(R_0)\setminus\{0\}$, and let 
\[
   u_j(r, \theta')=
   2^{(b_1\kappa+d_1)j}u(2^{\kappa j}r, 2^{j}\theta'), \quad j\ge 1, 
\]
%
%
%
where $b_1=n/s+\gamma_1+\alpha$, $d_1=(n-1)/s+\alpha$, and $\kappa$ is some $j$-independent constant to be determined later. 
Then $u_j\in C_c^{\infty}(R_j)$, where 
\[
   R_j:=\{x\in \mathbb{R}^n \mid 2^{-\kappa j}<r<2^{-\kappa j+1}, \ \ 2^{-j}\delta<|\theta'|<2^{-j+1}\delta\}.
\] 
  Denote 
\[
  I_0:=\||x|^{\gamma_1}|x'|^{\alpha}u\|_{L^s(R_0)}, \quad A_0:=\||x|^{\gamma_2}|x'|^{\mu}\nabla u\|_{L^p(R_0)}, \quad B_0:=\||x|^{\gamma_3}|x'|^{\beta}u\|_{L^q(R_0)}. 
\]
In the following, the notation $A\simeq B$ means $B/C\le A\le CB$, and $A\lesssim B$ means $A\le CB$, for some $C>1$ depending only on $s, p, q, a, \gamma_1, \gamma_2, \gamma_3, \alpha, \mu$ and  $\beta$.  

For any $\kappa\in\mathbb{R}$, we have 
\begin{equation}\label{eqNC_1_13}
 \||x|^{\gamma_1}|x'|^{\alpha}u_j\|_{L^s(R_j)}
    \simeq I_0. 
\end{equation}
Another computation gives 
\begin{equation}\label{eqNC_1_14'}
 \||x|^{\gamma_2}|x'|^{\mu}\nabla u_j\|_{L^p(R_j)}
   \lesssim 
   2^{(b_1\kappa+d_1-b_2\kappa-d_2)j}A_0,
   \end{equation}
   where $b_2=n/p+\gamma_2+\mu-1$ and $d_2=(n-1)/p+\mu-1$. Since we are in the case when $1/p+(\mu-1)/(n-1)\ne 1/q+\beta/(n-1)$ and $1/p+(\gamma_2+\mu-1)/n\ne 1/q+(\gamma_3+\beta)/n$,  we have, using (\ref{eqNCA_5}) and (\ref{eqNC_1_11}), that $b_1\ne b_2$ and $d_1\ne d_2$. Now we fix  
   \[
      \kappa:=\frac{d_2-d_1}{b_1-b_2}\in \mathbb{R}\setminus\{0\}, 
   \]
  so that 
   \begin{equation}\label{eqNC_1_14}
       \||x|^{\gamma_2}|x'|^{\mu}\nabla u_j\|_{L^p(R_j)}\lesssim A_0. 
   \end{equation}
   Using (\ref{eqNCA_5}), (\ref{eqNC_1_11}), and the definition of $\kappa$, we have 
    \begin{equation}\label{eqNC_1_15}
 \||x|^{\gamma_3}|x'|^{\beta}u_j\|_{L^q(R_j)}\simeq B_0. 
 \end{equation}
  For any positive integer $m$,  $w:=\sum_{j=1}^{m}u_j \in C_c^{\infty}(\mathbb{R}^n)$. Since $(\supp u_j)\cap (\supp u_i)=\phi$ for $i\ne j$,    we have,  by (\ref{eqNC_1_13}), (\ref{eqNC_1_14}) and (\ref{eqNC_1_15}), that
  \[
      \||x|^{\gamma_1}|x'|^{\alpha}w\|^s_{L^s(\mathbb{R}^n)}\simeq mI_0^s, \quad \||x|^{\gamma_2}|x'|^{\mu}\nabla w\|^p_{L^p(\mathbb{R}^n)}\lesssim mA^p_0, 
      \quad \||x|^{\gamma_3}|x'|^{\beta}w\|^q_{L^q(\mathbb{R}^n)}\simeq mB_0^q. 
  \]
  If $w$ satisfies (\ref{eqNC}), then we have, by the above, that 
 \[
    m^{1/s}\le C\frac{A_0^{a}B_0^{1-a}}{I_0}m^{a/p+(1-a)/q}. 
 \]
 Since $I_0, A_0, B_0>0$  and $m$ can be arbitrarily large, we have $1/s\le a/p+(1-a)/q$.  (\ref{eqNCA_7}) is proved. 
 \end{proof}


\noindent\emph{Proof of the necessity part of Theorem \ref{thmD_2}}:  Let $n\ge 1$, $s, p, q, a, \gamma_1, \gamma_2$ and $\gamma_3$ satisfy (\ref{eqNCA_1}) and (\ref{eqNCB_2}). We show that if (\ref{eqD_2_2}) holds for all $u\in C_c^\infty(\mathbb{R})$, then (\ref{eqNCB_3})-(\ref{eqNCB_5}) hold. This is the same as 
 in \cite{CKN} when $q\ge 1$, while the proof there applies to $q>0$ as well.

Since the formulation of our conditions is somewhat different from that in \cite{CKN}, we present 
a proof of the necessity of (\ref{eqNCB_3})-(\ref{eqNCB_5}) using similar arguments as in the proof of Lemma \ref{lemNC_1}. 
Condition (\ref{eqNCB_3}) follows from a dimensional analysis argument as in the proof of (\ref{eqNCA_5}) 
with $\alpha=\mu=\beta=0$.
Set $\alpha=\mu=\beta=0$ and $x_0=(0, ..., R)$ in the proof of (\ref{eqNCA_6_1}), the same arguments give (\ref{eqNCB_4}).

 To prove (\ref{eqNCB_5}), let $u=f_2(|x|)$ where $f_2$ is given by (\ref{eqNC_1_f2}) with $\alpha=0$, and insert $u$ into (\ref{eqD_2_2}). 
When $0<a<1$, we have $1/s+\gamma_1/n=1/p+(\gamma_2-1)/n=1/q+\gamma_3/n$. Similar to (\ref{eqNC_1_5})-(\ref{eqNC_1_7}), we have 
\[
   \||x|^{\gamma_1}u\|_{L^s(\mathbb{R})}\ge C\epsilon^{-1/s}, \quad \||x|^{\gamma_2}u'\|_{L^p(\mathbb{R})}\le C\epsilon^{-1/p}, \quad \||x|^{\gamma_3}u\|_{L^q(\mathbb{R})}\le C\epsilon^{-1/q}. 
\]
Using (\ref{eqD_2_2}) and the above, we have $\epsilon^{-1/s}\le C\epsilon^{-a/p-(1-a)/q}$ for arbitrarily small $\epsilon$, 
thus the  inequality in (\ref{eqNCB_5}) follows. 
In view of (\ref{eqNCB_3}),     we have $1/s+\gamma_1/n=1/q+\gamma_3/n$ when $a=0$, and   $1/s+\gamma_1/n=1/p+(\gamma_2-1)/n$ when $a=1$. The inequality in (\ref{eqNCB_5}) when $a=0$ or $a=1$ follows from the same proof for $0<a<1$. 
\qed


  \section{A nonlinear Poincar\'{e} inequality}\label{sec_3}
In this section, we give the proof of Theorem \ref{thm1-new}.

\bigskip

\noindent\emph{Proof of Theorem \ref{thm1-new}}: We divide the proof into three steps.

\bigskip

\noindent\textbf{Step 1.} We prove (\ref{est1-new}) under the hypotheses of the theorem. 

\medskip

For $\lambda=1$, Theorem \ref{thm1-new} is a generalized Poincare inequality (see e.g. Lemma 1.1.11 of \cite{Mazja}).  In the rest of Step 1 we assume $\lambda\ne 1$.

Since $C^1(\bar{\Omega})$ is dense in $W^{1, p}(\Omega)$, we may assume without loss of generality that $w\in C^1(\bar{\Omega})$ and $w>0$ in $\bar{\Omega}$. Let $u:=w^{1/\lambda}$, 
then inequality (\ref{est1-new}) takes an equivalent formulation: 
for  all $u\in C^1(\bar{\Omega})$ and $u>0$ in $\bar{\Omega}$, 
\begin{equation*}
\|v\|_{ L^p(\Omega) }\le C \|\nabla v\|_{ L^p(\Omega) }, \ \  \mbox{where}\ v: = u^\lambda- (\bar u)^\lambda \ \ \mbox{and}\  \bar u:= \avint_{S}u. 
\end{equation*}

We prove (\ref{est1-new}) by contradiction argument. 
Suppose the contrary, 
there exists a sequence of
positive functions  $\{u_j\}\in C^1(\overline \Omega)$ such that 
\begin{equation}\label{est2-new}
v_j:= (u_j)^\lambda- (\bar u_j)^\lambda
\end{equation}
satisfies
\begin{equation}\label{est2-new-1}
1= \|v_j\|_{ L^p(\Omega) } > j\|\nabla v_j\|_{ L^p(\Omega) },
\end{equation}
 where 
$
\bar u_j :=  \avint_{S}u_j>0.  
$

By (\ref{est2-new-1}) and the compact embedding of $W^{1,p}(\Omega)$ to $L^p(\Omega)$, there exists some $v\in W^{1, p}(\Omega)$ such that, after passing to  a subsequence (still denoted by $\{v_j\}$),  $v_j\rightharpoonup v$  in  $W^{1, p}(\Omega)$, $v_j\to v$ in $L^p(\Omega)$ and q.e.  
  in
$\Omega$, $\|\nabla v\|_{L^p(\Omega)}=0$, and $\|v\|_{L^p(\Omega)}=1$. 
 Now we have, using (\ref{est2-new-1}), that
\begin{equation}\label{eq6_2_1}
  \|v_j-v\|_{W^{1, p}(\Omega)}\to 0.
\end{equation}

We divide into two cases, $\lambda>1$ and $0<\lambda<1$. 

\bigskip

\noindent\emph{Case 1.} $\lambda>1$.

\medskip

In this case the function $s\to s^\lambda$ is convex, and therefore 
\begin{equation*}
\bar v_j\ge \big(\avint_Su_j\big)^{\lambda} 
-(\bar u_j)^\lambda=0.
\label{est5a}
\end{equation*}
Thus, by (\ref{eq6_2_1}),  we have  
$\bar{v}_j=\avint_{S}v_j\to v>0$.

Passing to another subsequence if necessary, we either have $\bar u_j\to \alpha \in [0, \infty)$ or $\bar u_j\to \infty$.

If $\bar u_j\to \alpha \in [0, \infty)$,  we have 
\[
u_j\to (v+\alpha^\lambda)^{1/\lambda}\ \  \mbox{a.e. in}\ \Omega.
\]
By Fatou's lemma, 
\[
|S| (v+\alpha^\lambda)^{1/\lambda}=
 \int_{ S }\liminf_{j\to \infty} u_j
\le \liminf_{j\to \infty}\int_{S } u_j=  \alpha |S|.
\]
A contradiction, since  $v>0$ and $\alpha \ge  0$. So inequality (\ref{est1-new}) holds.

In the rest of Case 1, we assume $\bar u_j\to \infty$.

Denote  $a_j: = \bar u_j \to \infty$,  and write 
\begin{equation}\label{est3-new}
0\le  u_j = a_j+\eta_j.
\end{equation}
Then
\begin{equation}\label{est4-new}
\int_{  S }\eta_j=0,\quad \forall\ j,
\end{equation}
and, by (\ref{est2-new}), 
\begin{equation*}
v_j= (a_j+\eta_j)^\lambda - (a_j)^\lambda. 
\end{equation*}
We will show that this leads to a contradiction.

\medskip

Write
$
v_j^+(\theta) = \max\{ v_j(\theta), 0\}$, 
$v_j^-(\theta)=\max\{ -v_j(\theta), 0\}$,  
 $\theta\in \overline \Omega$.
Then $v_j=v_j^+-v_j^-$.
By (\ref{eq6_2_1}) and the positivity of $v$, we have
\begin{equation}\label{lem1-new}
  \|v_j^{+}-v\|_{L^1(\Omega)}\to 0, \quad \|v_j^{-}\|_{L^1(\Omega)}\to 0. 
\end{equation}

\begin{lem} 
\begin{equation*}
(a_j)^{\lambda-1}  \int_{ \Omega } \eta_j^{-}\to 0,  \quad \textrm{and}\quad 
(a_j)^{\lambda-1}  \int_{ S } |\eta_j|\to 0.
\label{est6}
\end{equation*}
\label{lem1-3-new}
\end{lem}
\begin{proof}
Write
\[
v_j^-= (a_j)^\lambda- (a_j-\eta_j^-)^\lambda
= (a_j)^\lambda \big( 1-(1- \frac {\eta_j^-}{a_j})^\lambda \big),
\]
and recall from (\ref{est3-new}) that $0\le w^-\le a_j$. Since $\lambda\ge 1$, we have the following elementary inequality:
\[
g(t):= 1- (1-t)^\lambda - t\ge 0,\quad\forall\ 0\le t\le 1.
\]
Indeed, the above inequality holds due to the concavity of $g$ in
$[0, 1]$ ($g''(t)=-(\lambda-1) (1-t)^{ \lambda-2}<0$
for all $0<t<1$) and the fact that $g(0)=g(1)=0$.

Now we have, using (\ref{lem1-new}) and the above, that
\[
\circ(1)= \int_{ \Omega }  v_j^-= (a_j)^\lambda \Big( 1-\big(1- \frac {\eta_j^-}{a_j}\big)^\lambda \Big)
\ge  (a_j)^{\lambda}  \int_{  \Omega  } 
\frac { \eta_j^-} {a_j}=
 (a_j)^{\lambda-1} 
 \int_{ \Omega } \eta_j^-.
\]
Lemma \ref{lem1-3-new} follows from the above and (\ref{est4-new}).
\end{proof}

\begin{lem} There exists some positive constant $C$ independent of $j$
 such that
 $$
\int_{ S } (\eta_j^+)^\lambda  \ge \frac 1C, \quad \forall\ j.
$$
\label{lem1-4-new}
\end{lem}
\begin{proof}
We will use the following elementary inequality:
for $\lambda\ge 1$, there exists some positive constant $C$, depending 
only on $\lambda$,  such that
\[
(1+t)^\lambda-1 \le C(t^\lambda+t),\quad\forall \ t\ge 0.
\]
With this constant $C$, we have, using (\ref{lem1-new}), that 
\[
  \begin{split}
v|S|+\circ(1) & =  \int_{S } v_j^+
= (a_j)^\lambda  \int_{ S  }
\Big( \big(1+ \frac {\eta_j^+}{a_j}\big)^\lambda -1\Big)
\\
 & \le C  (a_j)^\lambda    \int_{  S  }
\Big(  \big(  \frac {\eta_j^+}{a_j}\big)^\lambda + 
  \frac {\eta_j^+}{a_j}\Big)
=  C \int_{ S  } 
\Big( (\eta_j^+)^\lambda +(a_j)^{\lambda-1} \eta_j^+\Big).
\end{split}
\]
Lemma \ref{lem1-4-new} follows from the above in view of 
Lemma \ref{lem1-3-new}.
\end{proof}

\begin{lem}
For every $\epsilon>0$,
$(a_j)^{\lambda-1}|\{\eta_j>\epsilon\}|\to 0$.
\label{lem1-5-new}
\end{lem}
\begin{proof}
Since
\[
v_j\ge (a_j+\epsilon)^\lambda- (a_j)^\lambda>0
\quad\mbox{on}\ \ \{\eta_j>\epsilon\},
\]
we have, using (\ref{lem1-new}), that 
\[
v|\Omega|+\circ(1)=  \int_{ \Omega }  v_j^+\ge
\int_{  \{\eta_j>\epsilon\}  } 
\left(  (a_j+\epsilon)^\lambda- (a_j)^\lambda \right)
=\left(  (a_j+\epsilon)^\lambda- (a_j)^\lambda \right) | \{\eta_j>\epsilon\}|.
\]
Lemma \ref{lem1-5-new} follows from the above since  $\lambda>1$ and $a_j\to \infty$.
\end{proof}

\begin{lem}
For every $\epsilon>0$,
$
  \int_{\Omega }
\left[ (\eta_j-\epsilon)^+\right]^\lambda\to 0$ as $j\to \infty$.
\label{lem1-6-new}
\end{lem}
\begin{proof}
 For  $\epsilon>0$,
denote $\xi_j:= \left[ (\eta_j-\epsilon)^+\right]^\lambda$.
By Lemma \ref{lem1-5-new},
$|\{ \xi_j=0\}|\to |\Omega|>0$  as $j\to \infty$. 
Apply a generalized Poincare inequality (see e.g.
Lemma 7.16 and Lemma 7.12 in \cite{GT}--writing $\Omega$ as the  union of finitely many convex open sets and apply these lemmas on each of the convex open sets) and use 
 (\ref{est2-new-1}) and the fact that $\lambda\ge 1$, we have
\[
  \begin{split}
 \int_{ \Omega } \xi_j & \le C  \int_{ \Omega } |\nabla \xi_j|
  \le
C   \int_{ \{\eta_j>\epsilon\} }  \left[(\eta_j-\epsilon)^+\right]^{\lambda-1}   |\nabla \eta_j^+|
\\
& \le C  \int_{ \{\eta_j>\epsilon\}}    (\eta_j^+)^{\lambda-1} |\nabla \eta_j^+|
\le C  \int_{ \{\eta_j>\epsilon\} }  |\nabla v_j|\to 0.
\end{split}
\]
Lemma \ref{lem1-6-new} is established.
\end{proof}

\bigskip

  For every $\epsilon>0$,
write $\eta_j=(\eta_j-\epsilon)+\epsilon\le (\eta_j-\epsilon)^+  + \epsilon$.
Thus
\[
(\eta_j^+)^\lambda \le 2^\lambda  \left[ (\eta_j-\epsilon)^+\right]^\lambda
+2^\lambda \epsilon^\lambda.
\]
It follows, using Lemma \ref{lem1-4-new} and Lemma \ref{lem1-6-new},  that
\[
0<\frac 1C\le   \int_{ S }
(\eta_j^+)^\lambda \le   
C  \int_{  S }   \left[ (\eta_j-\epsilon)^+\right]^\lambda
+2^\lambda \epsilon^\lambda |  S|
\le \circ(1)+ C\epsilon^{\lambda}.
\]
Sending $j$ to $\infty$, we have from the above that $0< 1/C \le C \epsilon^{\lambda}$.  Sending $\epsilon$ to $0$, we have $0< 1/C\le 0$, a contradiction.
 Estimate (\ref{est1-new}) is established in Case 1. 
 
  \bigskip

\noindent\emph{Case 2.} $0<\lambda<1$.

\medskip

  Recall that $p\ge n/ (1+n\lambda)$. 
Since $0<\lambda<1$, the function $s\to s^{\lambda}$ is concave, and we have
\begin{equation*}
\bar v_j\le \big(\avint_{S}u_j\big)^{\lambda}
-(\bar u_j)^\lambda=0.
\label{E}
\end{equation*}
Thus, by (\ref{eq6_2_1}), we have 
\begin{equation}\label{eq6_2_2}
\bar v_j=\avint_{S}v_j\to v<0.
\end{equation}

Fix a $\delta>0$ satisfying $1+\delta\le \min\{2, 1/\lambda\}$.
We will make use of the following elementary fact:  For $0<\lambda<1$,
there exists some positive constant $C$, depending only on $\lambda$ and $\delta$, such that
\begin{equation*}
\left| (1+t)^{ 1/\lambda } -1-\frac 1\lambda t\right|\le
C(|t|^{1+\delta} +|t|^{ 1/\lambda }),
\quad \forall\ -1\le t<\infty.
\label{F}
\end{equation*}

By (\ref{est2-new}),
\begin{equation}\label{G}
u_j =\left( v_j+ (\bar u_j)^\lambda \right)^{ 1/\lambda}. 
\end{equation}
Integrating the above over $S$ gives, with $C$ given by the one in
(\ref{F}), that 
\begin{equation}\label{eq6_2_3}
  \begin{split}
   0 &= \frac 1{  |S|  }\int_{S} \Big( \left[  v_j+ (\bar u_j)^\lambda \right] ^{ 1/\lambda}-\bar u_j\Big) = (\bar u_j) \frac 1{  |S|  }\int_{S} \Big(  \left[ 1+  (\bar u_j) ^{-\lambda} v_j  \right] ^{ 1/\lambda} -1\Big)\\
        & \le \frac 1\lambda (\bar u_j) ^{1-\lambda}\bar v_j +  \frac {C \bar u_j }
{\lambda|S|  } \int_{S} \Big( \left| (\bar u_j) ^{-\lambda} v_j\right|^{1+\delta}+ \left| (\bar u_j) ^{-\lambda} v_j\right|^{1/\lambda} \Big).
\end{split}
\end{equation}
Since $1+\delta\le 1/\lambda$, $W^{1,p}(\Omega)$ embeds into $L^{1/\lambda}(\Omega)$ and $L^{1+\delta}(\Omega)$ by the assumption on $p$. By this and (\ref{eq6_2_1}), we have 
\begin{equation}\label{eq6_2_4}
\|v_j-v\|_{ L^{1+\delta}(\Omega) }\le C\|v_j-v\|_{ L^{1/\lambda}(\Omega) }
\le C\|v_j-v\|_{W^{1, p}(\Omega)}\to 0. 
\end{equation}
We deduce from (\ref{eq6_2_3}), using 
(\ref{eq6_2_2}) and (\ref{eq6_2_4}), that 
\[
|v|+\circ(1)=-\bar v_j  \le C \Big( (\bar u_j) ^{-\delta\lambda}  \int_\Omega  |v_j|^{1+\delta}+   (\bar u_j) ^{\lambda-1}  \int_\Omega  |v_j|^{1/\lambda}\Big)
\le C  \left( (\bar u_j) ^{-\delta\lambda} +(\bar u_j) ^{\lambda-1} \right).  
\]
Since $v\ne 0$, we have the boundedness of  $\{\bar u_j\}$.
Passing to a subsequence,
$\bar u_j\to \alpha$ for some $\alpha\in [0, \infty)$.
 Integrating 
(\ref{G}) over $S$ and using (\ref{eq6_2_4}) 
and $\bar u_j\to\alpha$, we have 
\[
\alpha +\circ(1) 
= \avint_{S}
\Big( v_j+ (\bar u_j)^\lambda \Big)^{ 1/\lambda}
= \Big( v+ \alpha^\lambda \Big)^{ 1/\lambda} +\circ(1).
\]
It follows that $ \alpha = \left( v+ \alpha^\lambda \right)^{ 1/\lambda}$ which implies
that $v= 0$. 
 A contradiction.  Estimate (\ref{est1-new}) is established in Case 2. Step 1 is completed. 

\bigskip

\noindent\textbf{Step 2.} Inequality (\ref{est1-new}) does not hold if $0<\lambda<1$ and $0<p<  n/(1+n\lambda)$. 

\medskip

 For simplicity, we let $\Omega\subset \mathbb{R}^n$ be a bounded open set, and $S\subset \Omega$ has positive Lebesgue measure. Take a Lebesgue point $\bar{x}$ of $S$, i.e. $\lim_{r\to 0^+}|B_r(\bar{x})\cap S|/|B_r(\bar{x})|=1$. For convenience, $\bar{x}=0$ is the Lebesgue point. 
For  small $\epsilon>0$ and large $\alpha>1$, let
\[
v(x)=
\left\{
\begin{array}{ll}
 -1,  &   \mbox{if}\ |x|\ge \epsilon,\\  
 \displaystyle -1+\alpha \Big( 1- \frac {|x|}\epsilon\Big),  &   \mbox{if}\ 
|x|\le \epsilon.
\end{array}
\right.
\]
In the following, $C$ denotes some positive constant independent of $\alpha$ and $\epsilon$. 
A calculation gives
\begin{equation*}
\int_\Omega |v|^p=|\Omega\setminus B_{\epsilon}|+\int_{B_{\epsilon}}|v|^p\ge |\Omega|+\frac{1}{C}\alpha^p\epsilon^n, 
\end{equation*}
\[
\int_\Omega |\nabla v|^p =\alpha^p \epsilon^{n-p} |B_1|,
\]
\[
\int_{S} (v+1)^{1/\lambda}
= \alpha^{ 1/\lambda} \epsilon^n \int_{ \{|y|\le 1, \ \epsilon y\in S\}} (1-|y|)^{1/\lambda}dy, 
\]
where $|O(\alpha^p\epsilon^n)|\le C\alpha^p\epsilon^n$.  

Since $0$ is a Lebesgue point of $S$, we have 
\[
   \lim_{\epsilon\to 0^+}\frac{|\{|y|\le 1, \ \epsilon y\in S\}|}{\{|y|\le 1\}}=\lim_{\epsilon\to 0^+}\frac{|B_{\epsilon}(0)\cap S|}{|B_{\epsilon}(0)|}=1. 
\]
It follows that 
\[
    \lim_{\epsilon\to 0^+}\int_{\{|y|\le 1, \ \epsilon y\in S\}}(1-|y|)^{1/\lambda}dy=\int_{|y|\le 1}(1-|y|)^{1/\lambda}dy>0. 
\]
Now we fix the value of $\alpha$ so that
$\int_{S} (v+1)^{1/\lambda}= |S|$.  
 So $\alpha\le C \epsilon^{-n\lambda}$.

Consider
\[
    u:= (v+1)^{1/\lambda}.
\]
Then $\bar u=\avint_{S}u=1$, $u\ge 0$,  $v= u^\lambda- \bar u^{\lambda}$.
Using $p<n/(1+n\lambda)$, 
\[
   \int_\Omega |\nabla v|^p\le C\alpha ^p \epsilon^{n-p}\le C\epsilon^{ n-(1+n\lambda)p} \to 0.
\]
This and (\ref{eq6_2_3}) violate (\ref{est1-new}) for any choice of $C$. Step 2 is completed.


\bigskip

\noindent\textbf{Step 3.} Inequality (\ref{est1-new}) does not hold if $0<\lambda<\infty$ and $0<p<1$. 

\medskip

For simplicity, we take $\Omega=S=[-1, 1]^n$. For $\alpha>0$ small, let 
\[
   f(x_1):=\left\{
     \begin{array}{ll}
         |x_1|^{\alpha}, & x_1<0,\\
         -|x_1|^{\alpha}, & x_1\ge 0, 
     \end{array}
   \right.
\]
and  
\[
   w(x):=(2+f(x_1))^{\lambda}.
\]
Then $w\in W^{1, p}([-1, 1]^n)$ and $w\ge 1$. 

By the definition of $w$, we have $\avint_{[-1, 1]^n}w^{1/\lambda}=2$. Let $v=w-(\avint_{[-1, 1]^n}w^{1/\lambda})^{\lambda}$.  We have, 
   for  some constant $C>0$ depending only on $\lambda$ and $p$, that  
   \begin{equation}\label{eq6_2_6}
     \int_{[-1, 1]^n}|v(x)|^pdx\ge \int_{[1/2, 1]^n}|v(x)|^pdx=\int_{[1/2, 1]^n}|(2+|x_1|^{\alpha})^{\lambda}-2^{\lambda}|^pdx\ge \frac{1}{C}.  
   \end{equation}
   On the other hand, by the assumption that $0<p<1$, we have 
   \[
     \begin{split}
     \int_{[-1, 1]^n} |\nabla v|^pdx &   =  \int_{[-1, 1]^n}|\lambda(2+f(x_1))^{\lambda-1}f'(x_1)|^pdx \\
     &  \le  C\int_{-1}^{1}\alpha^p|x_1|^{(\alpha-1)p}dx_1\le C\alpha^{p} \to 0
     \end{split}
   \]
   as $\alpha\to 0$. 
   This and (\ref{eq6_2_6}) violate (\ref{est1-new}).  Step 3 is completed. Theorem \ref{thm1-new} is proved.
\qed

\bigskip


\noindent\emph{Proof of Corollary \ref{cor_new}}:
  If $q\le 1$, then,  by Theorem \ref{thm1-new} with $\lambda=1/q$, we have
\[
   \|w\|_{L^p(\Omega)}\le \big(\avint_{S}w^q\big)^{1/q}\cdot |\Omega|^{1/p}+ \|w-\big(\avint_{S}w^q\big)^{1/q}\|_{L^p(\Omega)}\le \big(\avint_Sw^{q}\big)^{1/q}|\Omega|^{1/p}+C \|\nabla w \|_{ L^p(\Omega) }. 
\]
If $q>1$, then (\ref{est1-newcor}) follows from the result for $q=1$ and H\"{o}lder's inequality.
 The corollary is proved.
\qed


  \section{Extension of the Caffarelli-Kohn-Nirenberg inequalities from $q\ge 1$ to $q>0$}\label{sec_4}
   
   In this section, we prove Theorem \ref{thmD_2}. The necessity part has been established in Section \ref{sec_2}. The sufficiency part follows from the following theorem which includes the  inequalities on cones. 
   
     Let
$
\mathbb{S}^{n-1}:= \{x\in \mathbb{R}^{n}\ |\ |x|=1\}
$
be the unit sphere in $\mathbb{R}^{n}$.
For  any $\Omega\subset\mathbb{S}^{n-1}$ with nonempty Lipchitz boundary, denote the cone
\begin{equation}\label{eq_cone}
   K:=\{rx\mid r\ge 0,  \    x\in \Omega\}. 
\end{equation}
    \begin{thm}\label{thmQ_2}
   Let  $n\ge 1$, $K=\mathbb{R}^n$ or $K$ be as above,    and  $s, p, q, \gamma_1, \gamma_2, \gamma_3, a$ satisfy  (\ref{eqNCA_1}) and  (\ref{eqNCB_2})-(\ref{eqNCB_5}). 
        Then there exists some positive constant $C$ 
     such that for all $u\in C^{0, 1}_c(\overline{K})$  
       \begin{equation}\label{eqQ2_1}
       \||x|^{\gamma_1}u\|_{L^s(K)}\le 
       C\||x|^{\gamma_2}\nabla u\|_{L^p(K)}^a\||x|^{\gamma_3}u\|_{L^q(K)}^{1-a}. 
  \end{equation}
  Furthermore, on any compact set in the parameter space in which (\ref{eqNCA_1}) and (\ref{eqNCB_2}) hold, 
    the constant $C$ is bounded. 
  \end{thm}

 \begin{lem}\label{lemQ_1}
   Let  $n\ge 1$, $0<r_1<r_2<\infty$, $K=\mathbb{R}^n$ or $K$ be given by (\ref{eq_cone}),       $s, p, q, a, \gamma_1, \gamma_2$ and $\gamma_3$ satisfy  (\ref{eqNCA_1}), (\ref{eqNCB_3}), (\ref{eqNCB_4}), 
      $1/s+\gamma_1/n>0$, and $1/s\le a/p+(1-a)/q$. 
     Then there exists some positive constant $C$, depending only on $s, q, a, \gamma_1, \gamma_2, \gamma_3$, $r_1, r_2$ and $\Omega$,  
     such that for all $u\in C^{0, 1}(K\cap B_{r_2})$, 
      \begin{equation}\label{eqQ_1}
       \||x|^{\gamma_1}u\|_{L^s(K\cap B_{r_1})}\le C \||x|^{\gamma_1}u\|_{L^s(K\cap B_{r_2}\setminus B_{r_1})}+
       C\||x|^{\gamma_2}\nabla u\|_{L^p(K\cap B_{r_2})}^a\||x|^{\gamma_3}u\|_{L^q(K\cap B_{r_1})}^{1-a}. 
  \end{equation}
  \end{lem}
\begin{proof}
For simplicity, we only prove (\ref{eqQ_1}) for $r_1=1$ and $r_2=2$. The general case can be proved similarly. 
 
     For $a=0$, we deduce from (\ref{eqNCB_3}), (\ref{eqNCB_4}) and $1/s\le a/p+(1-a)/q$      that $\gamma_1=\gamma_3$ and $s=q$, thus (\ref{eqQ_1}) is obvious. 
In the rest of the proof we assume $0<a\le 1$.  Without loss of generality, assume $u\ge 0$. 

\bigskip

\noindent\textbf{Step 1.} We  prove (\ref{eqQ_1}) for $p=1$ and $\gamma_1=0$. 

\medskip 

Let
\[
  R_k:=\{x\in K \mid 2^{k-1}\le |x|\le 2^{k}\}, \quad k\in \mathbb{Z}.   
 \]
Denote 
   \begin{equation*}
          A_k:=\int_{R_k}|u|^sdx,\quad 
  M_k:=\int_{R_k}||x|^{\gamma_2}\nabla u(x)|dx, \quad  N_k:=\int_{R_k}||x|^{\gamma_3}u|^qdx. 
   \end{equation*}

We first establish for any $0<\epsilon<2^{an}-1$ that 
 \begin{equation}\label{eqQ_2}
      A_k \le \theta A_{k+1}+C(M_{k}+M_{k+1})^{as}N_k^{(1-a)s/q}, \quad k\in \mathbb{Z}, 
   \end{equation}
   where 
   \begin{equation*}
      \theta:=\frac{a(1+\epsilon)}{2^{an}-(1+\epsilon)(1-a)}, 
   \end{equation*}
   and $C$ depends only on $s, q, a, \gamma_1, \gamma_2, \gamma_3$, $r_1, r_2, K$ and $\epsilon$. 
Since $K$ is a cone, by (\ref{eqNCB_3}) and scaling, we only need to prove (\ref{eqQ_2}) for $k=0$.

Let $\bar{u}=\avint_{R_1}u(y)dy$. For any $0<\epsilon<2^{an}-1$, $x\in R_0$ and $\xi\in R_1$, 
we have
\begin{equation}
\begin{split}
   |u(x)|^s & =|u(x)|^{(1-a)s}|u(x)|^{as}\\
   & \le |u(x)|^{(1-a)s}(|u(x)-\bar{u}|+|\bar{u}-u(\xi)|+|u(\xi)|)^{as}\\
   & \le (1+\epsilon) |u(x)|^{(1-a)s}|u(\xi)|^{as}+C|u(x)|^{(1-a)s}(|u(x)-\bar{u}|+|\bar{u}-u(\xi)|)^{as}.
   \end{split}
\end{equation} 
Taking $\avint_{R_1}\int_{R_0}\cdot dxd\xi$ of the above and using H\"{o}lder's inequality, we have 
\begin{equation}\label{eqQ_2_1}
   \begin{split}
     A_0 & \le (1+\epsilon)\avint_{R_1}|u(\xi)|^{as}d\xi\int_{R_0}|u(x)|^{(1-a)s}dx+C\int_{R_0}|u(x)|^{(1-a)s}|u(x)-\bar{u}|^{as}dx\\
     &+C\int_{R_0}|u(x)|^{(1-a)s}dx\avint_{R_1}|\bar{u}-u(\xi)|^{as}d\xi\\
     &\le (1+\epsilon)\frac{|R_0|^a}{|R_1|^a}\left(\int_{R_1}|u(\xi)|^sd\xi\right)^a\left(\int_{R_0}|u(x)|^sdx\right)^{1-a}+C\int_{R_0}|u(x)|^{(1-a)s}|u(x)-\bar{u}|^{as}dx\\
     &+C\int_{R_0}|u(x)|^{(1-a)s}dx\avint_{R_1}|\bar{u}-u(\xi)|^{as}d\xi\\
     & =:(1+\epsilon)\frac{|R_0|^a}{|R_1|^a}A_0^{1-a}A_1^a+C(I_1+I_2). 
   \end{split}
\end{equation}
Since $p=1$, by (\ref{eqNCB_3}) and (\ref{eqNCB_4}), we have $1/s\ge a(1-1/n)+(1-a)/q$. Since we are in the case $1/s\le a/p+(1-a)/q$,  we have $a(1-1/n)+(1-a)/q\le 1/s\le a+(1-a)/q$. Thus there exist some $1\le t\le n/(n-1)$ ($1\le t\le \infty$ when $n=1$) such that $1/s=a/t+(1-a)/q$. Then by H\"{o}lder's inequality, Sobolev inequality and Poincar\'{e}'s inequality, we have 
\begin{equation}\label{eqQ_3}
\begin{split}
   I_1 & \le C\|u-\bar{u}\|_{L^t(R_0\cup R_1)}^{as}\|u\|_{L^q(R_0)}^{(1-a)s}\\
   & \le C\big( \|u-\bar{u}\|_{L^1(R_0\cup R_1)}+\|\nabla(u-\bar{u})\|_{L^1(R_0\cup R_1)}\big)^{as}N_0^{(1-a)s/q}\\
   & \le C(M_0+M_1)^{as}N_0^{(1-a)s/q}. 
   \end{split}
\end{equation}
Similarly, we have 
\begin{equation}\label{eqQ_4}
\begin{split}
   I_2 & \le C\|u-\bar{u}\|_{L^t(R_1)}^{as}\|u\|_{L^q(R_0)}^{(1-a)s}\\
   & \le C\big( \|u-\bar{u}\|_{L^1(R_1)}+\|\nabla(u-\bar{u})\|_{L^1(R_1)}\big)^{as}N_0^{(1-a)s/q}\\
   & \le  CM_1^{as}N_0^{(1-a)s/q}. 
   \end{split}
\end{equation}
By (\ref{eqQ_2_1}), (\ref{eqQ_3}), (\ref{eqQ_4}) and the fact that $|R_0|/|R_1|=2^{-n}$,  we have, for any $0<\epsilon<2^{an}-1$, that
\[
    \begin{split}
  A_0 
  & \le (1+\epsilon)\frac{|R_0|^a}{|R_1|^a}A_0^{1-a}A_1^{a}+C(M_0+M_1)^{as}N_0^{(1-a)s/q}\\
   & \le (1+\epsilon)2^{-an}((1-a)A_0+aA_1)+C(M_0+M_1)^{as}N_0^{(1-a)s/q}. 
    \end{split}
\]
Thus 
\[
   A_0\le \frac{a(1+\epsilon)}{2^{an}-(1+\epsilon)(1-a)}A_1+C(M_0+M_1)^{as}N_0^{(1-a)s/q}.
\]
So (\ref{eqQ_2}) holds for $k=0$, and therefore holds for all $k\in \mathbb{Z}$.

Since $a>0$, $2^{an}>1$, and $0<\epsilon<2^{an}-1$, we have 
  $0<\theta<1$.  
 
  For $c, d\ge 0$, $c+d\ge 1$,  and sequences $x_n, y_n\ge 0$, $n\ge 1$, we have 
\begin{equation}\label{eqD_C}
   \sum_{n=1}^{\infty} x_n^cy_n^d\le \Big(\sum_{n=1}^{\infty} x_n\Big)^c\Big(\sum_{n=1}^{\infty} y_n\Big)^d. 
\end{equation}
 
 Take the sum of (\ref{eqQ_2}) over $k\le 0$, by the fact $1/s\le a+(1-a)/q$ and  (\ref{eqD_C}) with $c=(1-a)s/q$ and $d=as$, 
  we have, that
   \[
     \begin{split}
      \sum_{k\le 0}A_k & \le \theta A_1+\theta \sum_{k\le 0}A_k+C \sum_{k\le 0}(M_{k}+M_{k+1})^{as}N_k^{(1-a)s/q}\\
                                                      & \le \theta A_1+\theta \sum_{k\le 0}A_k+C\Big( \sum_{k\le 0}(M_{k}+M_{k+1})\Big)^{as}\Big( \sum_{k\le 0}N_k\Big)^{(1-a)s/q}. \\
     \end{split}
   \]
   So 
   \begin{equation}\label{eqQ_8}
      \begin{split}
         \int_{K\cap B_1}|u|^sdx & =
      \sum_{k\le 0}A_k   \le \frac{\theta}{1-\theta} A_1+C\Big(\sum_{k\le 0}(M_{k}+M_{k+1})\Big)^{as}\Big(\sum_{k\le 0}N_k\Big)^{(1-a)s/q}\\
                                                         & \le \frac{\theta}{1-\theta} \int_{R_1}|u|^sdx+C \left(\int_{K\cap B_2}||x|^{\gamma_2}\nabla u|dx\right)^{as}\left(\int_{K\cap B_1}||x|^{\gamma_3}u|^qdx\right)^{(1-a)s/q}. 
      \end{split}
   \end{equation}
  Thus when $p=1$ and $\gamma_1=0$, (\ref{eqQ_1}) follows from (\ref{eqQ_8}). 
 
 \bigskip
  
 \noindent \textbf{Step 2.} We prove (\ref{eqQ_1}) for $p=1$ and $\gamma_1\ne 0$. 
 
 \medskip
 
  We will reduce it to Step 1.  Make a change of variables $y=|x|^{\gamma_1 s/n}x$, and define $\tilde{u}(y):=u(x)$,   
  $\tilde{\gamma}_1=0$, $\tilde{\gamma}_2=(\gamma_2n+\gamma_1s(1-n))/(\gamma_1 s+n)$ and $\tilde{\gamma}_3=(\gamma_3 q-\gamma_1 s)n/(\gamma_1 s+n)q$. 
 We have $s, q>0$ from (\ref{eqNCA_1}) and $1/s+\tilde{\gamma}_1/n=1/s>0$. 
  By computation and using (\ref{eqNCB_3}), 
   \begin{equation}\label{eqQ_11}
   a\Big(1+\frac{\tilde{\gamma}_2-1}{n}\Big)+(1-a)\Big(\frac{1}{q}+\frac{\tilde{\gamma}_3}{n}\Big)=\frac{n}{\gamma_1 s+n}\Big(a\Big(1+\frac{\gamma_2-1}{n}\Big)+(1-a)\Big(\frac{1}{q}+\frac{\gamma_3}{n}\Big)\Big)=\frac{1}{s}.
   \end{equation}
   Next, by (\ref{eqNCB_3}) and (\ref{eqNCB_4}) with $p=1$, we have $1/s\ge a(1-1/n)-(1-a)/q$. Use this and (\ref{eqQ_11}), 
   we have 
   \begin{equation*}
      a\tilde{\gamma}_2+(1-a)\tilde{\gamma}_3=n\Big(\frac{1}{s}-a\Big(1-\frac{1}{n}\Big)-\frac{1-a}{q}\Big)\ge 0. 
   \end{equation*}
   So we have verified (\ref{eqNCA_1}), (\ref{eqNCB_3}), (\ref{eqNCB_4}), and $1/s+\tilde{\gamma}_1/n>0$ with $\tilde{\gamma}_1=0$.  
     By this and the fact that $1/s\le a/p+(1-a)/q$,  apply Step 1 to $\tilde{u}(y)$ and $\tilde{\gamma}_1, \tilde{\gamma}_3, \tilde{\gamma}_3$, we have
  \begin{equation}\label{eqQ_10}
     \|\tilde{u}\|_{L^s(K\cap B_1)}\le C \|\tilde{u}\|_{L^s(K\cap B_{R}\setminus B_1)}+
     C\||y|^{\tilde{\gamma}_2}\nabla\tilde{u}\|_{L^1(K\cap B_{R})}^a\||y|^{\tilde{\gamma}_3}\tilde{u}\|_{L^q(K\cap B_{1})}^{1-a}. 
   \end{equation}
   Since $\gamma_1 s/n+1>0$, we have, with $R:=2^{\gamma_1 s/n+1}>1$, that
    \[ 
      \begin{split}
         & \int_{K\cap B_1}||x|^{\gamma_1} u(x)|^sdx= \frac{n}{\gamma_1 s+n} \int_{K\cap B_1}|\tilde{u}(y)|^sdy, \\
         & \int_{K\cap B_2}||x|^{\gamma_2}\nabla u(x)|dx= \int_{K\cap B_{R}}||y|^{\tilde{\gamma}_2}\nabla\tilde{u}(y)|dy,\\
         &  \int_{K\cap B_{1}}||x|^{\gamma_3} u(x)|^qdx 
         =\frac{n}{\gamma_1 s+n}\int_{K\cap B_{1}}||y|^{\tilde{\gamma}_3}\tilde{u}(y)|^qdy.  
      \end{split}
   \]
   By (\ref{eqQ_10}) and the above, we have (\ref{eqQ_1}).

     \bigskip
     
     \noindent\textbf{Step 3.} We prove (\ref{eqQ_1}) for $p> 1$. 
     
     \medskip
     
     Let $\bar{s}, \bar{p}, \bar{q}, \bar{a}, \bar{\gamma}_1, \bar{\gamma}_2$ and $\bar{\gamma}_3$ be defined by   
  \begin{equation*}
    \begin{split}
    &   \frac{1}{\bar{s}}=\frac{1}{s}+\frac{1}{p'}, \quad \bar{p}=1, \quad \frac{1}{\bar{q}}=\frac{s}{\bar{s}q},  
    \quad \bar{a}=\frac{as}{(1-a)\bar{s}+as}, \\
    & \bar{\gamma}_1=\frac{\gamma_1s}{\bar{s}}, \quad \bar{\gamma}_2=\frac{\gamma_1s}{p'}+\gamma_2, \quad \bar{\gamma}_3=\frac{\gamma_3s}{\bar{s}}, \\
        \end{split}
   \end{equation*} 
   where $1/p+1/p'=1$.

It can be verified that $0<\bar{s}< s$, and $\bar{s}, \bar{p}, \bar{q}, \bar{a}, \bar{\gamma}_1, \bar{\gamma}_2, \bar{\gamma}_3$ satisfy (\ref{eqNCA_1}), (\ref{eqNCB_3}), (\ref{eqNCB_4}), $1/\bar{s}+\bar{\gamma}_1/n>0$, and $1/\bar{s}\le \bar{a}\bar{p}+(1-\bar{a})/\bar{q}$ (for detail of the verification, see Lemma \ref{lemPre2_1}). 
    So we can apply Step 2  
     to $|u|^{s/\bar{s}}$ to obtain,  using  H\"{o}lder's inequality and Young's inequality, that 
       \begin{equation*}
    \begin{split}
    \displaystyle
    &\quad  \||x|^{\gamma_1}u\|_{L^s(K\cap B_{1})}^{s/\bar{s}} 
      = \||x|^{\bar{\gamma}_1}|u|^{s/\bar{s}}\|_{L^{\bar{s}}(K\cap B_{1})}\\
     & \le C \||x|^{\bar{\gamma}_1}|u|^{s/\bar{s}}\|_{L^{\bar{s}}(K\cap B_2\setminus B_1)}+C\||x|^{\bar{\gamma}_2}\nabla |u|^{s/\bar{s}}\|_{L^1(K\cap B_2)}^{\bar{a}}\||x|^{\bar{\gamma}_3}|u|^{s/\bar{s}}\|_{L^{\bar{q}}(K\cap B_1)}^{1-\bar{a}}\\
     & \le C \||x|^{\gamma_1}u\|^{s/\bar{s}}_{L^{s}(K\cap B_2\setminus B_1)}+C\||x|^{\bar{\gamma}_2} |u|^{s/\bar{s}-1}|\nabla u|\|_{L^1(K\cap B_2)}^{\bar{a}}\||x|^{\gamma_3}u \|_{L^{q}(K\cap B_1)}^{(1-\bar{a})q/\bar{q}}\\
     & \le  C\||x|^{\gamma_1}u\|^{s/\bar{s}}_{L^{s}(K\cap B_2\setminus B_1)}+C\||x|^{\bar{\gamma}_2-\gamma_2} |u|^{s/\bar{s}-1}\|^{\bar{a}}_{L^{p'}(K\cap B_2)}\||x|^{\gamma_2}\nabla u\|_{L^p(K\cap B_2)}^{\bar{a}} 
     \||x|^{\gamma_3}u \|_{L^{q}(K\cap B_1)}^{(1-\bar{a})s/\bar{s}}\\
     &\le C  \||x|^{\gamma_1}u\|^{s/\bar{s}}_{L^{s}(K\cap B_2\setminus B_1)}+C\| |x|^{\gamma_1}u\|^{\bar{a}s/p'}_{L^{s}(K\cap B_2)}\||x|^{\gamma_2}\nabla u\|_{L^p(K\cap B_2)}^{\bar{a}}\||x|^{\gamma_3}u \|_{L^{q}(K\cap B_1)}^{(1-\bar{a})s/\bar{s}}\\
     & \le  C  \||x|^{\gamma_1}u\|^{s/\bar{s}}_{L^{s}(K\cap B_2\setminus B_1)}+ \displaystyle \frac{1}{2}\||x|^{\gamma_1}u\|_{L^s(K\cap B_{1})}^{s/\bar{s}}
      +C\Big(\||x|^{\gamma_2}\nabla u\|_{L^p(K\cap B_2)}^{\bar{a}}\\
      &\quad\cdot \||x|^{\gamma_3}u \|_{L^{q}(K\cap B_1)}^{(1-\bar{a})s/\bar{s}}\Big)^{1/(1-\bar{a}\bar{s}/p')}. 
     \end{split}
  \end{equation*}
  Inequality (\ref{eqQ2_1}) follows from the above and the definitions of $\bar{a}$ and $\bar{s}$. 
 Lemma \ref{lemQ_1} is proved.  
 \end{proof}
 
 \bigskip

\noindent\emph{Proof of Theorem \ref{thmQ_2}.}
Without loss of generality, we assume $u\ge 0$.  By (\ref{eqNCB_3}) and scaling, we may assume $\supp u\subset B_1$. 

For $a=0$, we deduce from (\ref{eqNCB_3}), (\ref{eqNCB_4}) and (\ref{eqNCB_5}) that $\gamma_1=\gamma_3$ and $s=q$, thus (\ref{eqQ2_1}) is obvious.  
In the rest of the proof we assume $0<a\le 1$.  

\bigskip

\noindent\emph{Case 1.} $\displaystyle 1/s\le a/p+(1-a)/q$. 

\medskip

 In this case, inequality (\ref{eqQ2_1}) follows from Lemma \ref{lemQ_1} with $r_1=1$ and $r_2=2$.

\bigskip

\noindent\emph{Case 2.} $\displaystyle 1/s> a/p+(1-a)/q$.

\medskip

Case 2 can be reduced to Case 1 by section (V) in \cite{CKN} - this reduction is the same for $q>0$ even though $q\ge 1$ was assumed in the paper. For reader's convenience, we include such an argument here.  

By (\ref{eqNCB_3}) and (\ref{eqNCB_5}),  $1/p+(\gamma_2-1)/n\ne 1/q+\gamma_3/n$. Thus there exist some positive constants $\lambda_1$ and $\lambda_2$, such that $\hat{u}(x)=\lambda_1u(\lambda_2 x)$ satisfies $\||x|^{\gamma_2}\nabla \hat{u}\|_{L^p(K)}=1$ and $\||x|^{\gamma_3}\hat{u}\|_{L^{q}(K)}=1$.  We claim that there exist some  $0\le a', a''\le 1$, such that  
 \begin{equation}\label{eqthmQ_2_1}
\begin{split}
   \||x|^{\gamma_1}\hat{u}\|_{L^s(K)} & \le C\left(\||x|^{\gamma_2}\nabla u\|_{L^p(K)}^{a'}\||x|^{\gamma_3}u\|_{L^q(K)}^{1-a'}+\||x|^{\gamma_2}\nabla u\|_{L^p(K)}^{a''}\||x|^{\gamma_3}u\|_{L^q(K)}^{1-a''}\right)\\
   & =2C\||x|^{\gamma_2}\nabla \hat{u}\|_{L^p(K)}^a\||x|^{\gamma_3}\hat{u}\|_{L^{q}(K)}^{1-a}. 
   \end{split}
\end{equation}
Then by scaling, we have that (\ref{eqQ2_1}) holds for $u$. 

  To see (\ref{eqthmQ_2_1}) when $n\ge 2$, 
     notice that by (\ref{eqNCA_1}), (\ref{eqNCA_2})-(\ref{eqNCB_5}), it can be directly verified that $s, p, q, a, \gamma_1, \gamma_2$ and $\gamma_3$ satisfy (\ref{eqNCA_1})-(\ref{eqNCA_7}) with $\alpha=\mu=\beta=0$.   Then (\ref{eqthmQ_2_1}) follows from  Lemma \ref{lemPre_3} with $\alpha=\mu=\beta=0$. 
If $n=1$, we can obtain (\ref{eqthmQ_2_1}) by the same proof as that of Lemma \ref{lemPre_3}, where we set $\alpha=\mu=\beta=0$ and choose $\alpha'=\alpha''=0$ there. 
  Theorem \ref{thmQ_2}  is proved. 
 \qed

   \bigskip
   
 \noindent  \emph{Proof of Theorem \ref{thmD_2}.}
The necessity part has been proved in Section \ref{sec_2}. 
  The sufficiency part follows from Theorem \ref{thmQ_2} with $K=\mathbb{R}^n$.  
  \qed


\section{Proof of the sufficiency part of Theorem \ref{thm_main}}\label{sec_5}

In this section, we prove the sufficiency part of Theorem \ref{thm_main}.

We first prove  the sufficiency part of Theorem \ref{thm_main} when $1/s\le a/p+(1-a)/q$.  We make use of  Theorem \ref{thmD_2} (rather, its variants Theorem \ref{thmQ_2} and Lemma \ref{lemQ_1})  and Theorem \ref{thm1-new}. 

\medskip

For $\delta, h>0$, denote  $B'_{\delta}=\{x'\in\mathbb{R}^{n-1}\mid |x'|\le \delta\}$,  $D_{\delta}^h=B'_{\delta}\times[0, h]$ and $D_{\delta}=D_{\delta}^1$. 

\begin{lem}\label{lemQ_3}
    Let $n\ge 2$, $0<\delta_1<\delta_2<\infty$, $h>0$,    $s, p, q, a, \alpha, \mu$ and $\beta$ satisfy (\ref{eqNCA_1})-(\ref{eqNCA_7}) with $\gamma_1=\gamma_2=\gamma_3=0$. 
   Then there exists some positive constant $C$, depending only on $s, p, q, a, \alpha, \mu, \beta, \sigma$, $\delta_1, \delta_2$ and $h$,
    such that for all $u\in C^{0, 1}(D_{\delta_2}^h)$
       \begin{equation}\label{eqQ_3_1}
      \||x'|^{\alpha}u\|_{L^s(D_{\delta_1}^h)}\le C\||x'|^{\alpha}u\|_{L^s(D_{\delta_2}^h\setminus D_{\delta_1}^h)}+C\||x'|^{\mu}\nabla u\|_{L^p(D_{\delta_2}^h)}^a\||x'|^{\beta}u\|_{L^{q}(D_{\delta_2}^h)}^{1-a}. 
   \end{equation}
\end{lem}
\begin{proof}
 Since $\gamma_1=\gamma_2=\gamma_3=0$, we deduce from (\ref{eqNCA_5}) and (\ref{eqNCA_6_3}) that 
 $1/s-a/p-(1-a)/q\ge (a(\mu-1)+(1-a)\beta-\alpha)/(n-1)=\frac{n}{n-1}(1/s-a/p-(1-a)/q)$, i.e. 
     \begin{equation}\label{eqQ_3_1_0}
 \frac{1}{s}\le \frac{a}{p}+\frac{1-a}{q}.  
\end{equation}

Let $x=(r', \theta', x_n)$ be the cylindrical coordinates where $r'=|x'|$ and $\theta'=x'/|x'|$. 
For simplicity, we only prove the lemma when $h=1$, $\delta_1=1$ and $\delta_2=2$. 
The general case can be proved similarly. 

 If $a=0$, by (\ref{eqNCA_5}), (\ref{eqNCA_6_2}) and (\ref{eqNCA_7}), we have $s=q$ and $\alpha=\beta$, and therefore (\ref{eqQ_3_1}) is obvious. In the rest of proof we assume $0<a\le 1$.

 \bigskip

\noindent\textbf{Step 1.}   We  prove inequality (\ref{eqQ_3_1}) when $p=1$.  

\medskip

By (\ref{eqQ_3_1_0}), we have, in view of $p=1$ and $a\le 1$, that 
     \begin{equation}\label{eqPre2_8}
      \frac{1}{s}-\frac{1-a}{q}\le a\le 1.
     \end{equation}
     
     \bigskip
    
  \noindent  \emph{Case 1.}  $1/s-(1-a)/q=1$.
    
    \medskip

By (\ref{eqPre2_8}), we have $a=1$ and $s=1$. Because of this, (\ref{eqNCA_5}), and the fact that $s=p=1$ and $\gamma_1=\gamma_2=\gamma_3=0$, we have $\alpha=\mu-1$. 
   Let 
      \begin{equation*}
         \hat{s}=1,\quad  \hat{p}=1, \quad \hat{q}=q, \quad  \hat{a}=1, \quad \hat{\gamma}_1=\alpha, \quad \hat{\gamma}_2=\alpha+1, \quad \hat{\gamma}_3=0. 
               \end{equation*}
      It is easy to verify that $\hat{s}, \hat{p}, \hat{q}, \hat{a}, \hat{\gamma}_1, \hat{\gamma}_2, \hat{\gamma}_3$ satisfy (\ref{eqNCA_1}), (\ref{eqNCB_2})-(\ref{eqNCB_5}) and       $1/\hat{s}\le \hat{a}/\hat{p}+(1-\hat{a})/\hat{q}$. Apply Lemma \ref{lemQ_1} to $u(\cdot, x_n)$ for each fixed $0\le x_n\le 1$, with $K=\mathbb{R}^{n-1}$ and $s, p, q, a, \gamma_1, \gamma_2, \gamma_3$ replaced by $\hat{s}, \hat{p}, \hat{q}, \hat{a}, \hat{\gamma}_1, \hat{\gamma}_2, \hat{\gamma}_3$, we have, with notation $\nabla'=\nabla_{x'}$, that
  \[
   \int_{B'_1}|x'|^{\alpha}|u(x', x_n)|dx'\le C\int_{B_{2}'\setminus B'_1}|x'|^{\alpha}|u(x', x_n)|dx'+C\int_{B_{2}'}|x'|^{\alpha+1}|\nabla' u(x', x_n)|dx'. 
      \]
           Integrate the above in $x_n$ on $[0, 1]$, we have (\ref{eqQ_3_1}) in this case, i.e.        \begin{equation*}
       \||x'|^{\alpha}u\|_{L^1(D_1)}\le C \||x'|^{\alpha}u\|_{L^1(D_{2}\setminus D_1)}+C\||x'|^{\alpha+1}\nabla'u\|_{L^1(D_{2})}.
    \end{equation*}

  \bigskip
\noindent\emph{Case 2.} $1/s-(1-a)/q < 1$. 

\medskip

Let
\begin{equation}\label{eqPre2_2_b}
  b=\frac{1}{a}\Big(\frac{1}{s}-\frac{1-a}{q}\Big), \quad \lambda=\frac{a(1-b)}{1-ab}. 
  \end{equation}
 Since $a>0$,  $b$ is well defined. In the definition of $\lambda$ above,  we have used the assumption that $ab=1/s-(1-a)/q<1$. 
   By (\ref{eqQ_3_1_0}) with $p=1$, we have $b\le 1$. By  (\ref{eqNCA_5}) and (\ref{eqNCA_6_2}), we have $1/s-(a(1/p-1/n)+(1-a)/q)=(a(\gamma_2+\mu)+(1-a)(\gamma_3+\beta)-(\gamma_1+\alpha))/n\ge 0$. 
Thus when $p=1$, we have $b=(1/s-(1-a)/q)/a\ge (n-1)/n$. So $(n-1)/n\le b\le 1$. Consequently, we have $0\le \lambda\le 1$ in view of  $0<a\le 1$.

  Let 
 \begin{equation}\label{eqPre2_2_0}
      \begin{split}
         & 
         \hat{a}=ab, \quad \hat{s}=s, \quad \hat{p}=1, \quad 
         \frac{1}{\hat{q}}=\lambda+\frac{1-\lambda}{q}, \\
         & \hat{\gamma}_1=\alpha, \quad \hat{\gamma}_2=\mu, \quad \hat{\gamma}_3=\lambda\mu+(1-\lambda)\beta. 
        \end{split}
      \end{equation}
   We have shown that $0<\hat{a}<1$. 
   Using (\ref{eqNCA_1})-(\ref{eqNCA_7}) and the assumption that  $1/s-(1-a)/q < 1$, 
   it can be verified that  $\hat{s}, \hat{p}, \hat{q}, \hat{a}, \hat{\gamma}_1, \hat{\gamma}_2, \hat{\gamma}_3$ satisfy (\ref{eqNCA_1}), (\ref{eqNCB_2})-(\ref{eqNCB_5}) with $\hat{p}=1$, $1/\hat{s}\le \hat{a}/\hat{p}+(1-\hat{a})/\hat{q}$ and with $n$ replaced by $n-1$. For the details of the verification, see Lemma \ref{lemPre2_2} 
   and its proof.

          Let $m=\min\{1, q, s\}$ and   $1<\delta<2$ be some fixed number,  set 
           \[
       v:=u-\big(\avint_{D_{2}\setminus D_1}u^m\big)^{1/m}. 
            \]
     Apply Lemma \ref{lemQ_1} with $p=1$
     to $v(\cdot, x_n)$ for each fixed $0\le x_n\le 1$, with $r_1=1, r_2=2$, 
     and $s, p, q, a, \gamma_1, \gamma_2, \gamma_3$ replaced by $\hat{s}, \hat{p}, \hat{q}, \hat{a}, \hat{\gamma}_1, \hat{\gamma}_2, \hat{\gamma}_3$,  
     we have 
       \begin{equation}\label{eqQ_3_2}
       \begin{split}
       & \quad \||x'|^{\alpha}v(\cdot, x_n)\|_{L^s(B_1')}\\
       & \le C\||x'|^{\alpha}v(\cdot, x_n)\|_{L^s(B_2'\setminus B_1')}+C\||x'|^{\mu}\nabla'v(\cdot, x_n)\|_{L^1(B_2')}^{ab}\||x'|^{\hat{\gamma}_3}v(\cdot, x_n)\|_{L^{\hat{q}}(B_1')}^{1-ab}. 
       \end{split}
      \end{equation}
    Using the definition of $\hat{\gamma}_3$, $\hat{q}$ and $\lambda$ in (\ref{eqPre2_2_b}) and (\ref{eqPre2_2_0}),  
    and the fact that $0\le \lambda\le 1$, 
  we apply H\"older's inequality to estimate the last term in (\ref{eqQ_3_2}) as follows. 
  \begin{equation}\label{eqQ_3_3}
    \begin{split}
    \||x'|^{\hat{\gamma}_3}v(\cdot, x_n)\|_{L^{\hat{q}}(B_1')}^{1-ab} & =\|||x'|^{\mu}v(\cdot, x_n)|^{\lambda}\cdot ||x'|^{\beta}v(\cdot, x_n)|^{1-\lambda}\|_{L^{\hat{q}}(B_1')}^{1-ab}\\
    &   \le \||x'|^{\mu}v(\cdot, x_n)\|_{L^1(B_1')}^{\lambda(1-ab)}\||x'|^{\beta}v(\cdot, x_n)\|_{L^q(B_1')}^{(1-\lambda)(1-ab)} \\
    &   \le \||x'|^{\mu}v(\cdot, x_n)\|_{L^1(B_1')}^{a(1-b)}\||x'|^{\beta}v(\cdot, x_n)\|_{L^q(B_1')}^{(1-a)}. 
     \end{split}
\end{equation}
Next, we estimate the term $\int_{B_{1}'}|x'|^{\mu}|v|dx'$ in the above. 
Notice  that 
  \[
    |v(x',x_n)|\le C\int_{0}^{1}|v_{x_n}(x', t)|dt+C\int_{0}^{1}|v(x', t)|dt, \quad \forall\ (x', x_n)\in D_2. 
  \]
 So, for each $x_n\in [0, 1]$, 
   we have
   \begin{equation}\label{eqQ_3_4}
    \int_{B_{1}'}|x'|^{\mu}|v(x',x_n)|dx'   \le C\int_{B_{1}'}\int_{0}^{1}|x'|^{\mu}|v_{x_n}(x', t)|dtdx'+C\int_{B_{1}'}\int_{0}^{1}|x'|^{\mu}|v(x', t)|dtdx'. 
  \end{equation}
  Applying Lemma \ref{lemQ_1} in dimension $n-1$, we have, for every $x_n$ in $[0, 1]$, that
  \[
     \int_{B_1'}|x'|^{\mu}|v(x', x_n)|dx'\le C\int_{B_2'\setminus B_1'}|x'|^{\mu}|v(x', x_n)|dx'+C\int_{B_2'}|x'|^{\mu+1}|\nabla' v(x', x_n)|dx'.  
       \]
  Integrating the above in $x_n$ over $[0, 1]$, and then inserting it into (\ref{eqQ_3_4}),  we have
  \begin{equation}\label{eqQ_3_3_1}
     \int_{B_{1}'}|x'|^{\mu}|v(x',x_n)|dx' \le C\big( \||x'|^{\mu}v\|_{L^1(D_{2}\setminus D_{1})}+\||x'|^{\mu}\nabla v\|_{L^1(D_{2})}\big). 
  \end{equation}

Putting (\ref{eqQ_3_2}), (\ref{eqQ_3_3}) and (\ref{eqQ_3_3_1}) together, we have  
 \[
  \begin{split}
   \||x'|^{\alpha}v(\cdot,x_n)\|^s_{L^s(B'_{1})}  & \le  C\||x'|^{\alpha}v(\cdot,x_n)\|^s_{L^s(B'_2\setminus B'_{1})}
       + C\||x'|^{\mu}  \nabla ' v(\cdot,x_n)\|_{L^1(B_{2}')}^{abs}  \\
        & \cdot\||x'|^{\beta}v(\cdot,x_n)\|_{L^q(B_{1}')}^{(1-a)s} \left(\||x'|^{\mu}\nabla v\|_{L^1(D_{2})}+\| |x'|^{\mu}v\|_{L^1(D_{2}\setminus D_{1})}\right)^{a(1-b)s}.
      \end{split}
  \]
Integrating the above in $x_n$ over $[0, 1]$, applying H\"older's inequality, and followed by Young's inequality,  we have, using   $abs+(1-a)s/q=1$, that    \[
   \begin{split}
           \||x'|^{\alpha}v\|^s_{L^s(D_1)}
            & \le  C\||x'|^{\alpha}v\|^s_{L^s(D_2\setminus D_{1})}+C\||x'|^{\mu}  \nabla ' v\|_{L^1(D_2)}^{abs} \||x'|^{\beta}v\|_{L^q(D_{1})}^{(1-a)s}
           \big(\||x'|^{\mu}\nabla v\|_{L^1(D_{2})}\\
           & \quad +\| |x'|^{\mu}v\|_{L^1(D_{2}\setminus D_{1})}\big)^{a(1-b)s}\\
           & \le C\||x'|^{\alpha}v\|^s_{L^s(D_2\setminus D_{1})}+C \big(\||x'|^{\mu}\nabla v\|_{L^1(D_{2})}+\||x'|^{\mu}  \nabla ' v\|_{L^1(D_2)}^b\\
           &\quad \cdot\| |x'|^{\mu}v\|^{1-b}_{L^1(D_{2}\setminus D_{1})}\big)^{as} 
             \||x'|^{\beta}v\|_{L^q(D_{1})}^{(1-a)s} \\
                                                                                        & \le C\||x'|^{\alpha}v\|^s_{L^s(D_2\setminus D_{1})}+C    \big(\||x'|^{\mu}\nabla v\|_{L^1(D_{2})} + \||x'|^{\mu} v\|_{L^1(D_{2}\setminus D_{1})}\big)^{as} \\
                                                                                        & \quad \cdot\||x'|^{\beta} v\|^{(1-a)s}_{L^{q}(D_{1})}.                                                                                             \end{split}
  \]
  By the definition of $v$ and the above, using $m\le s, q$, we have 
   \begin{equation}\label{eqQ_3_5}
   \begin{split}
   & \quad   \||x'|^{\alpha}u\|_{L^s(D_1)}\\
   & \le C\||x'|^{\alpha}u\|_{L^s(D_{2}\setminus D_{1})}+     C    \big(\||x'|^{\mu}\nabla u\|^a_{L^1(D_{2})}+ \| |x'|^{\mu}v\|^a_{L^1(D_{2}\setminus D_{1})}\big) \||x'|^{\beta} u\|^{1-a}_{L^{q}(D_{2})}. 
     \end{split}
  \end{equation}
  Since  $m\le 1$ and $1\le |x'|\le 2$ in $D_2\setminus D_1$, we apply 
  Theorem \ref{thm1-new} to obtain 
  \begin{equation}\label{eqQ_3_7}
      \| |x'|^{\mu}v\|_{L^1(D_{2}\setminus D_{1})}\le C\|v\|_{L^1(D_{2}\setminus D_{1})}\le C \|\nabla u\|_{L^1(D_{2}\setminus D_{1})}\le C \| |x'|^{\mu}\nabla u\|_{L^1(D_{2}\setminus D_{1})}.   \end{equation}
  By (\ref{eqQ_3_5}) and (\ref{eqQ_3_7}), we have 
  \[
     \||x'|^{\alpha}u\|_{L^s(D_1)}\le C\||x'|^{\alpha} u\|_{L^s(D_{2}\setminus D_{1})}+C\||x'|^{\mu}\nabla u\|_{L^1(D_{2})}^{a}\||x'|^{\beta} u\|_{L^{q}(D_{2})}^{1-a}. 
  \]
The lemma is proved for $p=1$.

\bigskip

\noindent\textbf{Step 2.} We prove inequality (\ref{eqQ_3_1}) when $p>1$. 

\medskip

Let $\bar{s}, \bar{p}, \bar{q}, \bar{a}, \bar{\alpha}, \bar{\mu}$ and $\bar{\beta}$ be defined by  
  \begin{equation*}
    \begin{split}
    &   \frac{1}{\bar{s}}=\frac{1}{s}+\frac{1}{p'}, \quad \bar{p}=1, \quad \frac{1}{\bar{q}}=\frac{s}{\bar{s}q},  
    \quad \bar{a}=\frac{as}{(1-a)\bar{s}+as}, \\
    & \bar{\alpha}=\frac{\alpha s}{\bar{s}}, \quad 
     \bar{\mu}=\frac{\alpha s}{p'}+\mu, \quad \bar{\beta}=\frac{\beta s}{\bar{s}}, 
          \end{split}
   \end{equation*} 
   where $1/p+1/p'=1$.  
   It can be verified that $0<\bar{s}< s$, and $\bar{s}, \bar{p}, \bar{q}, \bar{a}, \bar{\alpha}, \bar{\mu}, \bar{\beta}$ satisfy (\ref{eqNCA_1})-(\ref{eqNCA_7}) with $\gamma_1=\gamma_2=\gamma_3=0$ and $s, p, q, a, \alpha, \mu, \beta$ replaced by $\bar{s}, \bar{p}, \bar{q}, \bar{a}, \bar{\alpha}, \bar{\mu}, \bar{\beta}$ respectively. For the details of the verification, see Lemma \ref{lemPre2_1} and its proof. 
    For $u\in C^{0, 1}(D_2)$, 
    we have $|u|^{s/\bar{s}}\in C^{0, 1}(D_{2})$. 
    Apply (\ref{eqQ_3_1}) with $p=1$  to $|u|^{s/\bar{s}}$, we have,  using  H\"{o}lder's inequality and Young's inequality, that 
    \begin{equation*}
    \begin{split}
    &\quad   \||x'|^{\alpha}u\|_{L^s(D_1)}^{s/\bar{s}} 
     = \||x'|^{\bar{\alpha}}|u|^{s/\bar{s}}\|_{L^{\bar{s}}(D_1)}\\
     & \le C\||x'|^{\bar{\alpha}}|u|^{s/\bar{s}}\|_{L^{\bar{s}}(D_{2}\setminus D_1)}+C\||x'|^{\bar{\mu}}\nabla |u|^{s/\bar{s}}\|_{L^1(D_{2})}^{\bar{a}}\||x'|^{\bar{\beta}} |u|^{s/\bar{s}}\|_{L^{\bar{q}}(D_{2})}^{1-\bar{a}}\\
     & =  C\||x'|^{\alpha} u\|_{L^{s}(D_{2}\setminus D_1)}^{s/\bar{s}}+C\||x'|^{\bar{\mu}-\mu} |u|^{s/\bar{s}-1}\cdot |x'|^{\mu}|\nabla u|\|_{L^1(D_{2})}^{\bar{a}}\||x'|^{\beta}u \|_{L^{q}(D_{2})}^{(1-\bar{a})q/\bar{q}}\\
          &\le C\||x'|^{\alpha} u\|_{L^{s}(D_{2}\setminus D_1)}^{s/\bar{s}}+C\| |x'|^{\alpha}u\|^{\bar{a}s/p'}_{L^{s}(D_{2})}\||x'|^{\mu}\nabla u\|_{L^p(D_{2})}^{\bar{a}}\||x'|^{\beta}u \|_{L^{q}(D_{2})}^{(1-\bar{a})s/\bar{s}}\\
     & \le C\||x'|^{\alpha} u\|_{L^{s}(D_{2}\setminus D_1)}^{s/\bar{s}}+\frac{1}{2}\| |x'|^{\alpha}u\|^{s/\bar{s}}_{L^s(D_{2})}+C\Big(\||x'|^{\mu}\nabla u\|_{L^p(D_{2})}^{\bar{a}}\||x'|^{\beta}u \|_{L^{q}(D_{2})}^{(1-\bar{a})s/\bar{s}}\Big)^{1/(1-\bar{a}\bar{s}/p')}.
          \end{split}
  \end{equation*}
 Inequality (\ref{eqQ_3_1}) follows from the above and the definitions of $\bar{a}$ and $\bar{s}$.  Lemma \ref{lemQ_3} is proved. 
  \end{proof}
 
\begin{rmk}
In the proof of Lemma \ref{lemQ_3}, when $a=1$ or when $0<a<1$ and $1/s+1-1/p\le q/s$, we can use Theorem A and the classical Poincar\'{e}'s inequality instead of Theorem \ref{thmD_2} and  \ref{thm1-new}. 
\end{rmk}

For $0\le r_1< r_2\le \infty$ and $\epsilon>0$, let 
\begin{equation}\label{eqR_e}
K_{r_1, r_2, \epsilon}:=\{x\in\mathbb{R}^n\mid r_1\le |x|< r_2,\ \  |x'|\le \epsilon |x|\}. 
\end{equation}
\begin{lem}\label{lemQ_4}
     Let $n\ge 2$,  $0\le r_1< r_2\le \infty$, $0<\epsilon_1<\epsilon_2\le 1$, $K_{\epsilon_i}:=K_{r_1, r_2, \epsilon_i}$, $i=1, 2$, and let $s, p, q, a, \gamma_1, \gamma_2, \gamma_3, \alpha, \mu$  and $\beta$ be real numbers satisfying (\ref{eqNCA_1})-(\ref{eqNCA_7}) with 
      $1/s\le a/p+(1-a)/q$. Then there exists some positive constant $C$, depending only on $s, p, q, a, \gamma_1, \gamma_2,  \gamma_3, \alpha, \mu, \beta, \epsilon_1, \epsilon_2,  r_1$ and $r_2$, 
  such that for all $u\in C^{1}(\bar{K}_{\epsilon_2})$,  
  \begin{equation}\label{eqQ_4_1}
  \||x|^{\gamma_1}|x'|^{\alpha}u\|_{L^s(K_{\epsilon_1})}  \le C\||x|^{\gamma_1}|x'|^{\alpha}u\|_{L^s(K_{\epsilon_2}\setminus K_{\epsilon_1})}+C\||x|^{\gamma_2}|x'|^{\mu}\nabla u\|_{L^p(K_{\epsilon_2})}^{a}\||x|^{\gamma_3}|x'|^{\beta}u\|_{L^q(K_{\epsilon_2})}^{1-a}.
\end{equation}
Furthermore, on any compact set in the parameter space in which (\ref{eqNCA_1})-(\ref{eqNCA_3}) hold, the constant $C$ is bounded. 
\end{lem}
\begin{rmk}
  Consider more general cones $K_{r_1, r_2, \Omega}=\{rx\mid r_1\le r\le r_2,\  x\in \Omega\}$ for some open set $\Omega\subset\mathbb{S}^{n-1}$ with Lipschitz boundary. For open sets $\Omega_1\subset\bar{\Omega}_1\subset \Omega_2\subset\mathbb{S}^{n-1}$ with Lipschitz boundaries, Lemma \ref{lemQ_4} still holds with $K_{r_1, r_2, \epsilon_i}$ replaced by $K_{r_1, r_2, \Omega_i}$, $i=1, 2$. 
  \end{rmk}

\bigskip

\noindent\emph{Proof of Lemma \ref{lemQ_4}.}

  For  $\epsilon>0$, denote $K_{\epsilon}:=K_{r_1, r_2, \epsilon}$, and let $K_{\epsilon}^{+}:=K_{\epsilon}\cap\{x_n\ge 0\}$, $K_{\epsilon}^{-}:=K_{\epsilon}\cap\{x_n< 0\}$. We will only prove (\ref{eqQ_4_1}) with $K_{\epsilon_i}$, $i=1, 2$, replaced by $K_{\epsilon_i}^+$. The estimate on $K_{\epsilon_i}^{-}$ is similar.  

\bigskip

\noindent\emph{Case 1.} $0<r_1<r_2<\infty$.

\medskip

 In this case, there exists a diffeomorphism $y=\Phi(x)$ from $K_{\epsilon_i}^+$ to $D_{\epsilon_i}=B_{\epsilon_i}'\times [0, 1]$, satisfying $|y'|/C\le |x'|\le C|y'|$.  Let $\tilde{\mu}=\mu+\gamma_2-\gamma_1/a+(1-a)\gamma_3/a$. By (\ref{eqNCA_6_1}) we have $\tilde{\mu}\ge \mu$. 
  Notice we are in the case $1/s\le a/p+(1-a)/q$, it can be verified that  $s, p, q, a, \alpha, \tilde{\mu}, \beta$ satisfy (\ref{eqNCA_1})-(\ref{eqNCA_7}) with $\mu$ replaced by $\tilde{\mu}$ and $\gamma_1=\gamma_2=\gamma_3=0$. Applying Lemma \ref{lemQ_3} to $\hat{u}=u\circ \Phi^{-1}$, we have 
  \begin{equation*}
   \begin{split}
       \||y'|^{\alpha}\hat{u}\|_{L^s(D_{\epsilon_1})} & \le  C\||y'|^{\alpha}\hat{u}\|_{L^s(D_{\epsilon_2}\setminus D_{\epsilon_1})}+C\||y'|^{\tilde{\mu}}\nabla \hat{u}\|_{L^p(D_{\epsilon_2})}^a\||y'|^{\beta}\hat{u}\|_{L^{q}(D_{\epsilon_2})}^{1-a}\\
       & \le C\||y'|^{\alpha}\hat{u}\|_{L^s(D_{\epsilon_2}\setminus D_{\epsilon_1})}+C\||y'|^{\mu}\nabla \hat{u}\|_{L^p(D_{\epsilon_2})}^a\||y'|^{\beta}\hat{u}\|_{L^{q}(D_{\epsilon_2})}^{1-a}.  
       \end{split}
   \end{equation*}
   Inequality (\ref{eqQ_4_1}) follows immediately.
   
   \bigskip
   
  \noindent\emph{Case 2.} $r_1=0$ or $r_2=\infty$.
  
  \medskip
   
   Working with $u(\lambda x)$ instead of $u(x)$, we only need to treat the cases when $r_1 =0$  and $r_2=1$, or $r_1=1$ and $r_2=\infty$, or $r_1=0$ and $r_2=\infty$. 
   Let $R_k:=\{x\in\mathbb{R}^n\mid 2^{k-1}\le |x| < 2^{k}\}$, $k\in \mathbb{Z}$. By Case 1, (\ref{eqNCA_5}) and scaling, we have, for every $k\in \mathbb{Z}$, that 
    \begin{equation}\label{eqlem_in}
    \begin{split}
  \||x|^{\gamma_1}|x'|^{\alpha}u\|^s_{L^s(R_k\cap K_{\epsilon_1})} & \le C\||x|^{\gamma_1}|x'|^{\alpha}u\|^s_{L^s(R_k\cap K_{\epsilon_2}\setminus K_{\epsilon_1})}+C\||x|^{\gamma_2}|x'|^{\mu}\nabla u\|_{L^p(R_k\cap K_{\epsilon_2})}^{as}\\
  & \quad\cdot\||x|^{\gamma_3}|x'|^{\beta}u\|_{L^q(R_k\cap K_{\epsilon_2})}^{(1-a)s}. 
  \end{split}
\end{equation}
When  $r_1=0$ and $r_2=\infty$, take the sum of (\ref{eqlem_in}) over all $k\in \mathbb{Z}$, we have,  using  $as/p+\frac{(1-a)s}{q}\ge 1$ and 
 (\ref{eqD_C}), that  
\[
  \begin{split}
   & \quad  \||x|^{\gamma_1}|x'|^{\alpha}u\|^s_{L^s(K_{\epsilon_1})} \\
   & \le C\||x|^{\gamma_1}|x'|^{\alpha}u\|^s_{L^s(K_{\epsilon_2}\setminus K_{\epsilon_1})}+C\sum_{k=-\infty}^{\infty}\||x|^{\gamma_2}|x'|^{\mu}\nabla u\|_{L^p(R_k\cap K_{\epsilon_2})}^{as}\||x|^{\gamma_3}|x'|^{\beta}u\|_{L^q(R_k\cap K_{\epsilon_2})}^{(1-a)s}\\
    & \le C\||x|^{\gamma_1}|x'|^{\alpha}u\|^s_{L^s(K_{\epsilon_2}\setminus K_{\epsilon_1})}+C\||x|^{\gamma_2}|x'|^{\mu}\nabla u\|_{L^p(K_{\epsilon_2})}^{as}\||x|^{\gamma_3}|x'|^{\beta}u\|_{L^q(K_{\epsilon_2})}^{(1-a)s}.
    \end{split}
\]
So  (\ref{eqQ_4_1}) is proved when $r_1=0$ and $r_2=\infty$. 
Inequality (\ref{eqQ_4_1}) for $r_1=0$ and $r_2=1$ follows by summing (\ref{eqlem_in}) over $k\le 0$.  For $r_1=1$ and $r_2=\infty$, we sum (\ref{eqlem_in}) over $k\ge 0$. 
 Lemma \ref{lemQ_4} is proved.
\qed   

\bigskip

\noindent\emph{Proof of the sufficiency part of  Theorem \ref{thm_main} when $1/s\le a/p+(1-a)/q$.}

\medskip

Fix $\epsilon>0$ small, let $K_{\epsilon}$ be the cone defined by (\ref{eqR_e}) with $r_1=0$ and $r_2=\infty$. 
 
 By (\ref{eqNCA_1}), (\ref{eqNCA_3}), (\ref{eqNCA_5}), (\ref{eqNCA_6_2}) and (\ref{eqNCA_7}), we have that $s, p, q, a, \gamma_1+\alpha, \gamma_2+\mu, \gamma_3+\beta$ satisfy (\ref{eqNCA_1}) and (\ref{eqNCB_2})-(\ref{eqNCB_5}) with $\gamma_1, \gamma_2, \gamma_3$ replaced by $ \gamma_1+\alpha, \gamma_2+\mu, \gamma_3+\beta$ respectively. Then by Theorem \ref{thmQ_2}, we have 
 \begin{equation*}
     \||x|^{\gamma_1+\alpha}u\|_{L^s(\mathbb{R}^n\setminus K_{\epsilon})}\le C\||x|^{\gamma_2+\mu}\nabla u\|_{L^p(\mathbb{R}^n\setminus K_{\epsilon})}^{a}\||x|^{\gamma_3+\beta}u\|_{L^q(\mathbb{R}^n\setminus K_{\epsilon})}^{1-a}. 
  \end{equation*}
  Since $\epsilon |x|\le |x'|\le |x|$ for $x$ in $\mathbb{R}^n\setminus K_{\epsilon}$, we have 
  \begin{equation}\label{eqNC_out}
   \||x|^{\gamma_1}|x'|^{\alpha}u\|_{L^s(\mathbb{R}^n\setminus K_{\epsilon})}\le    C\||x|^{\gamma_2}|x'|^{\mu}\nabla u\|_{L^p(\mathbb{R}^n\setminus K_{\epsilon})}^{a}\||x|^{\gamma_3}|x'|^{\beta}u\|_{L^q(\mathbb{R}^n\setminus K_{\epsilon})}^{1-a}. 
\end{equation}
By Lemma \ref{lemQ_4}, 
\begin{equation}\label{eqin_1}
 \begin{split}
& \quad  \||x|^{\gamma_1}|x'|^{\alpha}u\|_{L^s(K_{\epsilon})} \\
&   \le C\||x|^{\gamma_1}|x'|^{\alpha}u\|_{L^s(K_{2\epsilon}\setminus K_{\epsilon})}+C\||x|^{\gamma_2}|x'|^{\mu}\nabla u\|_{L^p(K_{2\epsilon})}^{a}\||x|^{\gamma_3}|x'|^{\beta}u\|_{L^q(K_{2\epsilon})}^{1-a}. 
 \end{split}
 \end{equation}
  It follows from (\ref{eqNC_out}) and (\ref{eqin_1}) that 
   \begin{equation}\label{eqin_2}
  \||x|^{\gamma_1}|x'|^{\alpha}u\|_{L^s(\mathbb{R}^n)}\le C\||x|^{\gamma_2}|x'|^{\mu}\nabla u\|_{L^p(\mathbb{R}^n)}^{a}\||x|^{\gamma_3}|x'|^{\beta}u\|_{L^q(\mathbb{R}^n)}^{1-a}. 
 \end{equation}
 The sufficiency part of   Theorem \ref{thm_main} is proved when  $1/s\le a/p+(1-a)/q$. 
\qed

\bigskip

Next, we prove the sufficiency part of  Theorem \ref{thm_main} when $1/s> a/p+(1-a)/q$. We  reduce it to the case $1/s= a/p+(1-a)/q$ by the following lemma. This reduction procedure is analogous to the arguments in Section (V) in \cite{CKN}.


 \begin{lem}\label{lemPre_3}
 Let $n\ge 2$, $\Omega$ be a bounded open set in $\mathbb{R}^n$ and 
 $u\in C^{0, 1}(\Omega)$. Assume that for any $s, p, q, a, \gamma_1, \gamma_2, \gamma_3, \alpha, \mu, \beta$ satisfying (\ref{eqNCA_1})-(\ref{eqNCA_7}) and $1/s=a/p+(1-a)/q$,  there exists some constant $C$, depending only on $s, p, q, a, \gamma_1, \gamma_2, \gamma_3, \alpha, \mu$ and $\beta$, such that 
   \begin{equation}\label{eqPre3_0}
   \||x|^{\gamma_1}|x'|^{\alpha}u\|_{L^s(\Omega)}\le C'\||x|^{\gamma_2}|x'|^{\mu}\nabla u\|_{L^p(\Omega)}^{a}\||x|^{\gamma_3}|x'|^{\beta}u\|_{L^q(\Omega)}^{1-a}. 
   \end{equation} 
   Then for any  $s, p, q, a, \gamma_1, \gamma_2, \gamma_3, \alpha, \mu$ and $\beta$ satisfying (\ref{eqNCA_1})-(\ref{eqNCA_7}) with $1/s>a/p+(1-a)/q$, 
   there exists some constant $C$ and $0\le a', a''\le 1$, depending only on $s, p, q, a, \gamma_1, \gamma_2, \gamma_3$, $\alpha, \mu, \beta, \Omega$ and $C'$, such that  
   \begin{equation}\label{eqPre3_0'}
   \begin{split}
      \||x|^{\gamma_1}|x'|^{\alpha}u\|_{L^s(\Omega)} & \le C\Big(\||x|^{\gamma_2}|x'|^{\mu}\nabla u\|_{L^p(\Omega)}^{a'}\||x|^{\gamma_3}|x'|^{\beta}u\|_{L^q(\Omega)}^{1-a'}\\
      & \quad +\||x|^{\gamma_2}|x'|^{\mu}\nabla u\|_{L^p(\Omega)}^{a''}\||x|^{\gamma_3}|x'|^{\beta}u\|_{L^q(\Omega)}^{1-a''}\Big). 
      \end{split}
   \end{equation}
      \end{lem}
       \begin{proof}
        For  $u\in C^{0, 1}(\Omega)$,  we assume (\ref{eqPre3_0}) holds, and we will  prove (\ref{eqPre3_0'}). Let $C$ denote  a positive constant depending only on 
   $s, p, q, a, \gamma_1, \gamma_2, \gamma_3, \alpha, \mu, \beta, \Omega$ and $C'$    which may vary from line to line.  Condition (\ref{eqNCA_7}) and $1/s>a/p+(1-a)/q$ imply $0<a<1$. 
   Denote 
   $A:=\||x|^{\gamma_2}|x'|^{\mu}\nabla u\|_{L^p(\mathbb{R}^n)}$ and $B:=\||x|^{\gamma_3}|x'|^{\beta}u\|_{L^{q}(\mathbb{R}^n)}$. 
  
   For constants $0\le a', a''\le 1$, $\alpha', \alpha''$, we define $s', s'', \gamma_1',\gamma_1''$ by 
    \begin{equation}\label{eqPre3_2}
   \begin{split}
         & \frac{1}{s'}=\frac{a'}{p}+\frac{1-a'}{q}, \quad \gamma'_1+\alpha'=a'(\gamma_2+\mu-1)+(1-a')(\gamma_3+\beta),\\
         &  \frac{1}{s''}=\frac{a''}{p}+\frac{1-a''}{q}, \quad \gamma''_1+\alpha''=a''(\gamma_2+\mu-1)+(1-a'')(\gamma_3+\beta). 
         \end{split}
   \end{equation}
   Let $\zeta(x)$ be a smooth function satisfying $\zeta(x)=1$ for $|x|\le 1$,  $\zeta(x)=0$ for $|x|\ge 2$ and $|\nabla \xi(x)|\le 3$.  We have
     \begin{equation}\label{eqPre3_7}
     \||x|^{\gamma_1}|x_1|^{\alpha}u\|_{L^s(\mathbb{R}^n)}\le \||x|^{\gamma_1}|x_1|^{\alpha}\zeta u\|_{L^s(\mathbb{R}^n)}+\||x|^{\gamma_1}|x_1|^{\alpha}(1-\zeta)u\|_{L^s(\mathbb{R}^n)} =: I_1+I_2. 
  \end{equation}
       We estimate 
  \begin{equation}\label{eqPre3_8}
    I_1\le  \||x|^{\gamma'_1}|x_1|^{\alpha'}u\|_{L^{s'}(\mathbb{R}^n)}\left(\int_{|x|\le 2}\left||x|^{\gamma_1-\gamma'_1}|x'|^{\alpha-\alpha'}\right|^{ss'/(s'-s)}\right)^{1/s-1/s'}, 
  \end{equation}
  and 
  \begin{equation}\label{eqPre3_9}
    I_2\le  \||x|^{\gamma''_1}|x_1|^{\alpha''}u\|_{L^{s''}(\mathbb{R}^n)}\left(\int_{|x|\ge 1}\left||x|^{\gamma_1-\gamma''_1}|x'|^{\alpha-\alpha''}\right|^{ss''/(s''-s)}\right)^{1/s-1/s''}. 
  \end{equation}
  by H\"{o}lder's inequality, provided 
    \begin{equation}\label{eqPre3_3}
    \frac{1}{s'}< \frac{1}{s} \ \ \mbox{and}\ \  \frac{1}{s''}< \frac{1}{s}. 
      \end{equation}
  The second integrals on the right hand sides in (\ref{eqPre3_8}) and (\ref{eqPre3_9}) are finite if 
  \begin{equation}\label{eqPre3_4}
    \frac{1}{s'}+\frac{\gamma'_1+\alpha'}{n}<\frac{1}{s}+\frac{\gamma_1+\alpha}{n}     < \frac{1}{s''}+\frac{\gamma''_1+\alpha''}{n}, 
  \end{equation}
  \begin{equation}\label{eqPre3_5}
     \frac{1}{s'}+\frac{\alpha'}{n-1}<\frac{1}{s}+\frac{\alpha}{n-1} \ \ \mbox{and}\ \   \frac{1}{s''}+\frac{\alpha''}{n-1}<\frac{1}{s}+\frac{\alpha}{n-1}.  
  \end{equation}
    By the assumption of the lemma, we will have 
   \begin{equation}\label{eqPre3_18}
      \||x|^{\gamma'_1}|x_1|^{\alpha'}u\|_{L^{s'}(\mathbb{R}^n)}\le CA^{a'} B^{1-a'}, \quad \||x|^{\gamma''_1}|x_1|^{\alpha''}u\|_{L^{s''}(\mathbb{R}^n)}\le CA^{a''}B^{1-a''},
      \end{equation}
      provided 
      (\ref{eqNCA_1})-(\ref{eqNCA_7}) with $s, a, \gamma_1, \alpha$ there replaced by $s', a', \gamma'_1, \alpha'$ or $s'', a'',  \gamma_1'', \alpha''$ respectively. 
  
  So by (\ref{eqPre3_7})-(\ref{eqPre3_9}) and (\ref{eqPre3_18}), to prove (\ref{eqPre3_0'}), we only need to choose appropriate $a', a'', \alpha'$ and $\alpha''$ such that (\ref{eqPre3_3})-(\ref{eqPre3_5}) are satisfied,  and (\ref{eqNCA_1})-(\ref{eqNCA_7}) hold with $s, a, \gamma_1, \alpha$ there replaced by $s', a', \gamma'_1, \alpha'$ or $s'', a'',  \gamma_1'', \alpha''$ respectively. 
  
     The choice of $a'$ and $\alpha'$ and the choice of $a''$ and $\alpha''$ can be made independently and analogously. We always require $a'$ and $a''$ to be close to $a$ and in particular $0< a', a''< 1$.  
     By (\ref{eqPre3_2}),      conditions (\ref{eqNCA_1}), (\ref{eqNCA_5}) and (\ref{eqNCA_7}) always hold with $s, a, \gamma_1, \alpha$ there replaced by $s', a', \gamma'_1, \alpha'$ or $s'', a'',  \gamma_1'', \alpha''$ respectively. 
     By (\ref{eqPre3_2}), we have 
     \[
        a'(\gamma_2+\mu)+(1-a')(\gamma_3+\beta)-(\gamma_1'+\alpha')=a'.  
             \]
     By the above requirement on $a'$ and $a''$, we have (\ref{eqNCA_6_2})  with $s, a, \gamma_1, \alpha$ there replaced by $s', a', \gamma'_1, \alpha'$ respectively. Similarly,      we have (\ref{eqNCA_6_2})  with $s, a, \gamma_1, \alpha$ there replaced by $s'', a'',  \gamma_1'', \alpha''$ respectively.
          
     By  (\ref{eqPre3_2}), we have 
     \[
        \frac{1}{s'}+\frac{\gamma_1'+\alpha'}{n}=a'(\frac{1}{p}+\frac{\gamma_2+\mu-1}{n})+(1-a')(\frac{1}{q}+\frac{\gamma_3+\beta}{n}). 
     \]
     By (\ref{eqNCA_3}) and (\ref{eqNCA_5}), the right hand side of the above is strictly positive when $a'=a$. Thus  as long as we choose $a'$ close enough to $a$, we have $1/s'+(\gamma_1'+\alpha')/n>0$, and therefore (\ref{eqNCA_3}) holds with $s, a, \gamma_1, \alpha$ there replaced by $s', a', \gamma'_1, \alpha'$ respectively. Similarly, we have (\ref{eqNCA_3})  with $s, a, \gamma_1, \alpha$ there replaced by  $s'', a'',  \gamma_1'', \alpha''$ respectively, as long as  we choose $a''$ close enough to $a$.

     Moreover, by the assumption  $1/s>a/p+(1-a)/q$ and the definition of $s'$ and $s''$ in (\ref{eqPre3_2}), we have that (\ref{eqPre3_3}) hold as long as $a'$ and $a''$ are close enough to $a$.  
    By (\ref{eqNCA_7}), (\ref{eqNCA_5}) and the assumption that $1/s>a/p+(1-a)/q$, we have  $1/p+(\gamma_2+\mu-1)/n\ne 1/q+(\gamma_3+\beta)/n$.  For (\ref{eqPre3_4}) to hold,  we only need to require 
  \begin{equation*}
     \begin{split}
        & 0< a'<a<a''< 1, \quad \textrm{ if } \frac{1}{p}+\frac{\gamma_2+\mu-1}{n}>\frac{1}{q}+\frac{\gamma_3+\beta}{n},\\
        &1> a'>a>a''> 0, \quad \textrm{ if } \frac{1}{p}+\frac{\gamma_2+\mu-1}{n}<\frac{1}{q}+\frac{\gamma_3+\beta}{n}.
     \end{split}
  \end{equation*}
  
   
  It remains to show that we can further require $a', a'', \alpha', \alpha''$ to satisfy additional properties, such that (\ref{eqPre3_5}) is satisfied, and (\ref{eqNCA_2}), (\ref{eqNCA_6_1}) and (\ref{eqNCA_6_3}) hold with $s, a, \gamma_1, \alpha$ there replaced by $s', a', \gamma'_1, \alpha'$ or $s'', a'',  \gamma_1'', \alpha''$ respectively. 
  By the definition of $1/s'$ and $1/s''$ in (\ref{eqPre3_2}),  equation (\ref{eqPre3_5}) holds provided
  \begin{equation}\label{eqPre3_20}
        \alpha'<G(a'), \quad \alpha''<G(a''),
  \end{equation}
  where $G(\theta)=(n-1)(1/s-\theta/p-(1-\theta)/q)+\alpha$. 
  By the definition of $1/s'$ and $1/s''$ in (\ref{eqPre3_2}), equation (\ref{eqNCA_2}) holds with $s, a, \gamma_1, \alpha$ there replaced by $s', a', \gamma'_1, \alpha'$ or $s'', a'',  \gamma_1'', \alpha''$ respectively, provided
  \begin{equation}\label{eqPre3_21}
     \alpha'>F_1(a'), \quad \alpha''>F_1(a'),
  \end{equation}
  where $F_1(\theta)=-(n-1)(\theta/p+(1-\theta)/q)$. 
  By the definition of $\gamma_1'+\alpha'$ and $\gamma_1''+\alpha''$ in (\ref{eqPre3_2}), 
  equation (\ref{eqNCA_6_1}) holds with $s, a, \gamma_1, \alpha$ there replaced by $s', a', \gamma'_1, \alpha'$ or $s'', a'',  \gamma_1'', \alpha''$ respectively, provided
  \begin{equation}\label{eqPre3_22}
     \alpha'>F_2(a'), \quad \alpha''>F_2(a''),
       \end{equation}
  where $F_2(\theta)=\theta(\mu-1)+(1-\theta)\beta$. 
  By the definition of $1/s'$ and $1/s''$ in (\ref{eqPre3_2}), equation  (\ref{eqNCA_6_3}) holds with $s, a, \gamma_1, \alpha$ there replaced by $s', a', \gamma'_1, \alpha'$ or $s'', a'',  \gamma_1'', \alpha''$ respectively, provided (\ref{eqPre3_22}). So we only need to further require $a', a'', \alpha', \alpha''$ to satisfy (\ref{eqPre3_20})-(\ref{eqPre3_22}).
  
  By (\ref{eqNCA_2}), $1/s+\alpha/(n-1)>0$, so we have  $F_1(a)<G(a)$.  By (\ref{eqNCA_6_3}), (\ref{eqNCA_7}) and the assumption that $1/s>a/p+(1-a)/q$, the inequality in (\ref{eqNCA_6_3}) is strict, and therefore $F_2(a)<G(a)$. So as long as $a'$ and $a''$ are close enough to $a$, we can find $\alpha'$ and  $\alpha''$ to satisfy
  (\ref{eqPre3_20})-(\ref{eqPre3_22}). Lemma \ref{lemPre_3} is proved. 
  \end{proof}

\bigskip

\noindent\emph{Proof of the sufficiency part of  Theorem \ref{thm_main} when $1/s>a/p+(1-a)/q$.}

\medskip

 In this case, by (\ref{eqNCA_5}) and (\ref{eqNCA_7}), we must have $1/p+(\gamma_2+\mu-1)/n\ne 1/q+(\gamma_3+\beta)/n$. So there exist some constants $C$ and $\lambda$, such that $\hat{u}=Cu(\lambda x)$ satisfies $\||x|^{\gamma_2}|x'|^{\mu}\nabla \hat{u}\|_{L^p(\mathbb{R}^n)}=1$ and $\||x|^{\gamma_3}|x'|^{\beta}\hat{u}\|_{L^{q}(\mathbb{R}^n)}=1$.  
By Theorem \ref{thm_main} for $1/s\le a/p+(1-a)/q$, (\ref{eqNC}) holds for $\hat{u}$ and all $s, p, q, a, \gamma_1, \gamma_2, \gamma_3, \alpha, \mu, \beta$ satisfying (\ref{eqNCA_1})-(\ref{eqNCA_7}) and $1/s\le a/p+(1-a)/q$.  
 Then by Lemma \ref{lemPre_3}, when $1/s> a/p+(1-a)/q$, we have 
 \[
   \||x|^{\gamma_1}|x'|^{\alpha}\hat{u}\|_{L^p(\mathbb{R}^n)}\le C\||x|^{\gamma_2}|x'|^{\mu}\nabla \hat{u}\|_{L^p(\mathbb{R}^n)}^a\||x|^{\gamma_3}|x'|^{\beta}\hat{u}\|_{L^{q}(\mathbb{R}^n)}^{1-a}. 
\]
Then (\ref{eqNC}) holds for $u$ by scaling.  
\qed

\bigskip

The sufficiency part of Theorem \ref{thm_main} is proved.

\section{Two variants of Theorem \ref{thm_main} and Theorem A}\label{sec_6}

We have the following variant of Theorem \ref{thm_main}.

\begin{thm}\label{cor_in}
     Let $n\ge 2$,  $0\le r_1< r_2\le \infty$, $\epsilon>0$, $K:=K_{r_1, r_2, \epsilon}$ be defined as (\ref{eqR_e}), and $s, p, q, a, \gamma_1, \gamma_2, \gamma_3, \alpha, \mu$ and $\beta$ be real numbers satisfying (\ref{eqNCA_1})-(\ref{eqNCA_7}). 
          Then 
     there exists some positive constant $C$, depending only on $\epsilon, s, p, q, a, \gamma_1, \gamma_2, \gamma_3, \alpha, \mu$ and $\beta$, 
  such that for all $u\in C^{1}(\bar{K})$ with $u=0$ on $\partial K$, 
 \begin{equation}\label{eqNC_in}
  \||x|^{\gamma_1}|x'|^{\alpha}u\|_{L^s(K)}  \le C\||x|^{\gamma_2}|x'|^{\mu}\nabla u\|^a_{L^p(K)}\||x|^{\gamma_3}|x'|^{\beta}u\|_{L^q(K)}^{1-a}. 
    \end{equation}
Furthermore, on any compact set in the parameter space in which (\ref{eqNCA_1})-(\ref{eqNCA_3}) hold, the constant $C$ is bounded. 
\end{thm}
\begin{proof}
Extend $u$ to be zero outside $K$. When $1/s\le a/p+(1-a)/q$, apply Lemma \ref{lemQ_4} to $u$ with $K_{\epsilon_1}=K$ and $K_{\epsilon_2}$ be a larger cone containing $K$, we obtain (\ref{eqNC_in}). 

Now we consider the case when $1/s>  a/p+(1-a)/q$.
 By (\ref{eqNCA_5}) and (\ref{eqNCA_7}), we have $1/p+(\gamma_2+\mu-1)/n\ne 1/q+(\gamma_3+\beta)/n$. So there exist some constants $C$ and $\lambda$, such that $\hat{u}=Cu(\lambda x)$ satisfies $\||x|^{\gamma_2}|x'|^{\mu}\nabla \hat{u}\|_{L^p(K)}=1$ and $\||x|^{\gamma_3}|x'|^{\beta}\hat{u}\|_{L^{q}(K)}=1$.  
 Since we have proved (\ref{eqNC_in}) when $1/s= a/p+(1-a)/q$, we can apply Lemma \ref{lemPre_3} to $\hat{u}$ to obtain, for some $0\le a',  a''\le 1$, that  
 \[
\begin{split}
  &  \quad \||x|^{\gamma_1}|x'|^{\alpha}\hat{u}\|_{L^s(K)}\\
   & \le C\Big(\||x|^{\gamma_2}|x'|^{\mu}\nabla u\|_{L^p(K)}^{a'}\||x|^{\gamma_3}|x'|^{\beta}u\|_{L^q(K)}^{1-a'}+\||x|^{\gamma_2}|x'|^{\mu}\nabla u\|_{L^p(K)}^{a''}\||x|^{\gamma_3}|x'|^{\beta}u\|_{L^q(K)}^{1-a''}\Big)\\
   & =2C\||x|^{\gamma_2}|x'|^{\mu}\nabla \hat{u}\|_{L^p(K)}^a\||x|^{\gamma_3}|x'|^{\beta}\hat{u}\|_{L^{q}(K)}^{1-a}. 
   \end{split}
\]
Inequality (\ref{eqNC_in}) follows. 
\end{proof}

  The following is a variant of Theorem A.
    \begin{thm}\label{thm6_1}
     Let  $n\ge 1$, $R>0$, $B_R=\{x\in\mathbb{R}^n\mid |x|\le R\}$, 
     $0<\lambda<\infty$, 
     Assume 
     $s, p, q, a, \gamma_1, \gamma_2, \gamma_3$ satisfy (\ref{eqNCA_1}), (\ref{eqNCB_2})-(\ref{eqNCB_5}). Moreover, assume $1\le p\le \infty$   if    $1\le \lambda<\infty$, and $\max\{1, (n-1)/(1+(n-1)\lambda)\}\le p\le \infty$ if $0<\lambda<1$.  
     Then there exists some positive constant $C$, depending only on $s, p, q, a,  \gamma_1, \gamma_2, \gamma_3$ and $\lambda$,  
     such that for every nonnegative $w\in W^{1, 1}(B_R)$, $v:=w-(\avint_{\partial B_{|x|}}w^{1/\lambda})^{\lambda}$ satisfies
      \begin{equation}\label{eqPre3_1_0'}
  \||x|^{\gamma_1}v\|_{L^s(B_R)}\le C\||x|^{\gamma_2}\nabla v\|_{L^p(B_R)}^a\||x|^{\gamma_3}v \|_{L^q(B_R)}^{1-a}. 
  \end{equation}
  Furthermore, on any compact set in the parameter space in which (\ref{eqNCA_1}) and (\ref{eqNCB_2}) hold, 
  the constant $C$ is bounded. 
 \end{thm}
 \begin{proof}
   Let $C$ denote a positive constant depending only on $s, p, q, a, \gamma_1, \gamma_2, \gamma_3$ and $\lambda$, which may vary from line to line. 

 For $a=0$, we deduce from (\ref{eqNCB_3}), (\ref{eqNCB_4}) and (\ref{eqNCB_5}) that $\gamma_1=\gamma_3$ and $s=q$, thus (\ref{eqPre3_1_0'}) is obvious. In the rest of the proof we assume $0<a\le 1$.

\bigskip

\noindent\textbf{Case 1.} 
$1/s\le a/p+(1-a)/q$.

\medskip

Let 
 \[
    R_k:=\{x'\in B_1\mid \frac{1}{2^k}\le |x'|\le \frac{1}{2^{k-1}}\},     \quad k\in \mathbb{Z}.
 \]
 We first prove that 
 \begin{equation}\label{eqPre3_1_2_0}
 \||x|^{\gamma_1} v\|_{L^{s}(R_k)}\le C\||x|^{\gamma_2} \nabla v\|^{a}_{L^{p}(R_k)}\||x|^{\gamma_3} v\|^{1-a}_{L^{q}(R_k)}, \quad k\in \mathbb{Z}.
 \end{equation}
By scaling, using (\ref{eqNCB_3}), we only need to prove (\ref{eqPre3_1_2_0}) for $k=1$.

Let $\bar{s}, \bar{q}$ and $\bar{a}$ be defined as in (\ref{eqPre0_0}). Since $a>0$, we have $\bar{a}>0$. Define $t\in (0, \infty]$ by 
\[
  \frac{1}{t}=\frac{1}{\bar{a}}(\frac{1}{\bar{s}}-\frac{1-\bar{a}}{\bar{q}}), \quad \textrm{ if }\frac{1}{\bar{s}}-\frac{1-\bar{a}}{\bar{q}}>0,
\]
and $t=\infty$ if $1/\bar{s}-(1-\bar{a})/\bar{q}=0$. 

In the current case we have $1/s\le a/p+(1-a)/q$. By (\ref{eqNCB_3}) and (\ref{eqNCB_4}), we have $1/s\ge a(1/p-1/n)+(1-a)/q$. By the same arguments as in part (g) in the proof of Lemma \ref{lemPre2_1}, 
we have $1/\bar{s}\le \bar{a}+(1-\bar{a})/\bar{q}$ and $1/\bar{s}-(1-\bar{a})/\bar{q}\ge \bar{a}(n-1)/n\ge 0$, and therefore $(n-1)/n\le 1/t\le 1$. 
We have proved that 
\[
  1\le t\le \frac{n}{n-1} \textrm{ for }n\ge 2\quad \textrm{ and }1\le t\le \infty\textrm{ for }n=1. 
\]
By H\"older's inequality, provided $1/\bar{s}=\bar{a}/t+(1-\bar{a})/\bar{q}$, $0< \bar{a}\le 1$, $1\le t\le \infty$,  and $\bar{q}>0$, we have, 
\begin{equation}\label{eqPre3_1_2_1}
   \|v\|_{L^s(R_1)}^{s/\bar{s}}=\||v|^{s/\bar{s}}\|_{L^{\bar{s}}(R_1)}\le \||v|^{s/\bar{s}}\|^{\bar{a}}_{L^t(R_1)}\||v|^{s/\bar{s}}\|^{1-\bar{a}}_{L^{\bar{q}}(R_1)}= \||v|^{s/\bar{s}}\|^{\bar{a}}_{L^t(R_1)}\|v\|^{(1-\bar{a})s/\bar{s}}_{L^{q}(R_1)}. 
\end{equation}
where we have used the definition of $\bar{q}$ in the last step. 

Since $1\le t\le n/(n-1)$, we apply H\"older's inequality and Sobolev's inequality to obtain   
\begin{equation}\label{eqPre3_1_2_3}
\begin{split}
   \||v|^{s/\bar{s}}\|_{L^t(R_1)} & \le C\||v|^{s/\bar{s}}\|_{L^{\frac{n}{n-1}}(R_1)}\le C(\||v|^{s/\bar{s}}\|_{L^1(R_1)}+\|\nabla |v|^{s/\bar{s}}\|_{L^1(R_1)})\\
   & \le C\left(\| |v|^{s/\bar{s}-1}\|_{L^{p'}(R_1)}\| v\|_{L^p(R_1)}+\| |v|^{s/\bar{s}-1}\|_{L^{p'}(R_1)}\| \nabla v\|_{L^p(R_1)}\right)\\
   & \le C \|v\|_{L^s}^{s/p'}(\| v\|_{L^p(R_1)}+\| \nabla v\|_{L^p(R_1)}), 
   \end{split}
\end{equation}
where in the last step we have used the fact that $(s/\bar{s}-1)p'=s$ from the definition of $\bar{s}$. 
 
Since $p\ge 1$ when $1\le \lambda<\infty$ and $\max\{1, n/(1+n\lambda)\}\le p\le \infty$ when $0<\lambda<1$, we have, by  Theorem \ref{thm1-new},  that 
\[
       \|v\|_{L^p(R_1)}^p\le \int_{1/2}^{1} \|v\|_{L^p(\partial B_{\rho})}^pd\rho\le C\int_{1/2}^{1} \|\nabla_{tan}v\|_{L^p(\partial B_{\rho})}^pd\rho \le \|\nabla  v\|^p_{L^p(R_1)}. 
   \]
   By (\ref{eqPre3_1_2_1}), (\ref{eqPre3_1_2_3}) and the above, we have 
\begin{equation*}
   \|v\|^{s/\bar{s}}_{L^{s}(R_1)}\le  \||v|^{s/\bar{s}}\|^{\bar{a}}_{L^t(R_1)}\|v\|^{(1-\bar{a})s/\bar{s}}_{L^{q}(R_1)} 
   \le C\|v\|_{L^s}^{\bar{a}s/p'}\|\nabla v\|^{\bar{a}}_{L^p(R_1)}\|v\|^{(1-\bar{a})s/\bar{s}}_{L^q(R_1)}.
\end{equation*}
Using the definition of $\bar{a}$ and $\bar{s}$ in (\ref{eqPre0_0}), we deduce from the above that 
\[
    \|v\|_{L^{s}(R_1)}
   \le C\|\nabla v\|^{a}_{L^p(R_1)}\|v\|^{1-a}_{L^q(R_1)}.
\]
We have proved (\ref{eqPre3_1_2_0}) for $k=1$.

   Since we are in the case $as/p+(1-a)s/q\ge 1$, we can use  (\ref{eqD_C}) to deduce from (\ref{eqPre3_1_2_0}) that 
\[
  \begin{split}
   \sum_{k=-\infty}^{\infty}\int_{R_k}||x|^{\gamma_1}v|^{s}dx 
  & \le C \sum_{k=-\infty}^{\infty}\left(\int_{R_k}||x|^{\gamma_2}\nabla v|^pdx\right)^{as/p}\left(\int_{R_k}||x|^{\gamma_3}v|^qdx'\right)^{(1-a)s/q}\\
            & \le C \left(\sum_{k=-\infty}^{\infty}\int_{R_k}||x|^{\gamma_2}\nabla v|^pdx\right)^{as/p}\left(\sum_{k=-\infty}^{\infty} \int_{R_k}||x|^{\gamma_3}v|^qdx\right)^{(1-a)s/q}\\
            & \le C\||x^{\gamma_2} \nabla v\|^{as}_{L^{p}(B_R)}\||x|^{\gamma_3} v\|^{(1-a)s}_{L^{q}(B_R)}.
   \end{split}
\]
We have proved (\ref{eqPre3_1_0'}) in Case 1.

\bigskip

\noindent\textbf{Case 2.} $1/s> a/p+(1-a)/q$.

\medskip

By (\ref{eqNCB_3}) and (\ref{eqNCB_5}),  we have $1/p+(\gamma_2-1)/n\ne 1/q+\gamma_3/n$. Thus there exist some positive constants $\lambda_1$ and $\lambda_2$, such that $\hat{v}(x)=\lambda_1v(\lambda_2 x)$ satisfies $\||x|^{\gamma_2}\nabla \hat{v}\|_{L^p(K)}=1$ and $\||x|^{\gamma_3}\hat{v}\|_{L^{q}(K)}=1$.  
  By Case 1,   arguing as in the paragraph below (\ref{eqthmQ_2_1}), we can find $0\le a', a''\le 1$ such that  
  \[
\begin{split}
   \||x|^{\gamma_1}\hat{v}\|_{L^s(K)} & \le C\left(\||x|^{\gamma_2}\nabla v\|_{L^p(K)}^{a'}\||x|^{\gamma_3}v\|_{L^q(K)}^{1-a'}+\||x|^{\gamma_2}\nabla v\|_{L^p(K)}^{a''}\||x|^{\gamma_3}v\|_{L^q(K)}^{1-a''}\right)\\
   & =2C\||x|^{\gamma_2}\nabla \hat{v}\|_{L^p(K)}^a\||x|^{\gamma_3}\hat{v}\|_{L^{q}(K)}^{1-a}. 
   \end{split}
\]
Inequality (\ref{eqPre3_1_0'}) follows. 
\end{proof}


\section{Appendix: some facts about the parameters}\label{sec_A}

In this section, we prove some properties of the parameters $s, p, q, a, \gamma_1, \gamma_2, \gamma_3, \alpha, \mu$ and $\beta$ which we use in earlier sections. 

Let $s, p, q, a, \gamma_1, \gamma_2, \gamma_3, \alpha, \mu$ and $\beta$ be real numbers satisfying (\ref{eqNCA_1}), define $\bar{s}, \bar{p},  \bar{q}, \bar{a}$, $\bar{\gamma}_1, \bar{\gamma}_2, \bar{\gamma}_3, \bar{\alpha}, \bar{\mu}$ and $ \bar{\beta}$ by
    \begin{equation}\label{eqPre0_0}
    \begin{split}
    &   \frac{1}{\bar{s}}=\frac{1}{s}+\frac{1}{p'}, \quad \bar{p}=1, \quad \frac{1}{\bar{q}}=\frac{s}{q\bar{s}},  
    \quad \bar{a}=\frac{as}{(1-a)\bar{s}+as}, \\
    & \bar{\gamma}_1=\frac{\gamma_1s}{\bar{s}}, \quad \bar{\gamma}_2=\frac{\gamma_1s}{p'}+\gamma_2, \quad \bar{\gamma}_3=\frac{\gamma_3s}{\bar{s}}, \\
    & \bar{\alpha}=\frac{\alpha s}{\bar{s}}, \quad 
     \bar{\mu}=\frac{\alpha s}{p'}+\mu, \quad \bar{\beta}=\frac{\beta s}{\bar{s}}, 
          \end{split}
   \end{equation} 
   where $1/p+1/p'=1$.  Clearly,  $0<\bar{s}< s$.

   \begin{lem}\label{lemPre2_1}
 (i) If $n\ge 1$, $s, p, q, a, \gamma_1, \gamma_2$ and $\gamma_3$ satisfy (\ref{eqNCA_1}), (\ref{eqNCB_2})-(\ref{eqNCB_5}), then $\bar{s}, \bar{p},  \bar{q}, \bar{a}, \bar{\gamma}_1, \bar{\gamma}_2$ and $\bar{\gamma}_3$ also satisfy (\ref{eqNCA_1}) and (\ref{eqNCB_2})-(\ref{eqNCB_5}). 

(ii) If $n\ge 2$, $s, p, q, a, \gamma_1, \gamma_2, \gamma_3, \alpha, \mu$ and $\beta$  satisfy (\ref{eqNCA_1})-(\ref{eqNCA_7}), then 
  $\bar{s}, \bar{p},  \bar{q}, \bar{a}, \bar{\gamma}_1$, $\bar{\gamma}_2, \bar{\gamma}_3, \bar{\alpha}, \bar{\mu}$ and  $\bar{\beta}$ also satisfy (\ref{eqNCA_1})-(\ref{eqNCA_7}). 
  
  (iii)  Assume (\ref{eqNCA_1}) holds, then $1/s\le a/p+(1-a)/q$ if and only if  $1/\bar{s}\le \bar{a}/\bar{p}+(1-\bar{a})/\bar{q}$, and $1/s\ge a(1/p-1/n)+(1-a)/q$ if and only if $1/\bar{s}\ge \bar{a}(1-1/n)+(1-\bar{a})/\bar{q}$.
  \end{lem}
  \begin{proof}
  For convenience, denote $\Lambda=(s, p, q, a, \gamma_1, \gamma_2, \gamma_3, \alpha, \mu, \beta)$ and $\bar{\Lambda}=(\bar{s}, \bar{p},  \bar{q}, \bar{a}, \bar{\gamma}_1$, $\bar{\gamma}_2, \bar{\gamma}_3, \bar{\alpha}, \bar{\mu},  \bar{\beta})$. 
  By (\ref{eqPre0_0}), it is clear that $\bar{\Lambda}$ satisfies (\ref{eqNCA_1}).   
       
       Now we prove the following statements (a)-(h), which imply (i)-(iii).   
            
       (a) If $n\ge 2$ and (\ref{eqNCA_2}) holds for $\Lambda$, then (\ref{eqNCA_2}) also holds for $\bar{\Lambda}$.
       
       This follows from 
               \begin{equation*}
     \begin{split}
        & \frac{1}{\bar{s}}+\frac{\bar{\alpha}}{n-1}=\frac{s}{\bar{s}}\Big(\frac{1}{s}+\frac{\alpha}{n-1}\Big), \\
        & \frac{1}{\bar{p}}+\frac{\bar{\mu}}{n-1}=\frac{1}{p}+\frac{\mu}{n-1}+\frac{s}{p'}\Big(\frac{1}{s}+\frac{\alpha}{n-1}\Big), \\
         & \frac{1}{\bar{q}}+\frac{\bar{\beta}}{n-1}=\frac{s}{\bar{s}}\Big(\frac{1}{q}+\frac{\beta}{n-1}\Big),
         \end{split}
     \end{equation*}
     
        (b) If $n\ge 1$ and (\ref{eqNCA_3}) holds for $\Lambda$, then (\ref{eqNCA_3}) also holds  for $\bar{\Lambda}$.
        
       This follows from 
     \begin{equation}\label{eqPre1_6}
       \begin{split}
         &   \frac{1}{\bar{s}}+\frac{\bar{\gamma}_1+\bar{\alpha}}{n}=\frac{s}{\bar{s}}\Big(\frac{1}{s}+\frac{\gamma_1+\alpha}{n}\Big), \\
               &   \frac{1}{\bar{p}}+\frac{\bar{\gamma}_2+\bar{\mu}}{n}=\frac{1}{p}+\frac{\gamma_2+\mu}{n}+\frac{s}{p'}\Big(\frac{1}{s}+\frac{\gamma_1+\alpha}{n}\Big), \\
                        & \frac{1}{\bar{q}}+\frac{\bar{\gamma}_3+\bar{\beta}}{n}=\frac{s}{\bar{s}}\Big(\frac{1}{q}+\frac{\gamma_3+\beta}{n}\Big). 
       \end{split}
     \end{equation}

    (c) Let $n\ge 1$, then $\Lambda$ satisfies (\ref{eqNCA_5}) if and only if $\bar{\Lambda}$ satisfies (\ref{eqNCA_5}).

    Using the definition of $\bar{s}$, $\bar{\gamma}_1$ and $\bar{\alpha}$, we have
\begin{equation}\label{eqPre1_V_1}
  \begin{split}
    \frac{1}{\bar{s}}+\frac{\bar{\gamma}_1+\bar{\alpha}}{n}         & =\frac{s}{\bar{s}}\Big(1+\frac{as}{p'}\Big)^{-1}\Big(1+\frac{as}{p'}\Big)\Big(\frac{1}{s}+\frac{\gamma_1+\alpha}{n}\Big)\\
        & =\frac{s}{\bar{s}}\Big(1-a+\frac{s}{\bar{s}}a\Big)^{-1}\left(\frac{1}{s}+\frac{\gamma_1+\alpha}{n}+a\frac{s}{p'}\Big(\frac{1}{s}+\frac{\gamma_1+\alpha}{n}\Big)\right).
        \end{split}
\end{equation}
By the definition of $\bar{\gamma}_2$, $\bar{\gamma}_3$, $\bar{\mu}$ and $\bar{\beta}$, we have 
\begin{equation}\label{eqPre1_V_2}
  \begin{split}
   & \bar{a}\Big(1+\frac{\bar{\gamma}_2+\bar{\mu}-1}{n}\Big)+(1-\bar{a})\Big(\frac{1}{\bar{q}}+\frac{\bar{\gamma}_3+\bar{\beta}}{n}\Big)\\
   &=
    \frac{s}{\bar{s}}\Big(1-a+\frac{s}{\bar{s}}a\Big)^{-1}\left(a\Big(\frac{1}{p}+\frac{\gamma_2+\mu-1}{n}\Big)+(1-a)\Big(\frac{1}{q}+\frac{\gamma_3+\beta}{n}\Big)+a\frac{s}{p'}\Big(\frac{1}{s}+\frac{\gamma_1+ \alpha}{n}\Big)\right).  
    \end{split}        
\end{equation}
By (\ref{eqPre1_V_1}) and (\ref{eqPre1_V_2}), we have (c).

    
  (d) Let $n\ge 1$, then  $\Lambda$ satisfies (\ref{eqNCA_6_1}) if and only if  $\bar{\Lambda}$ satisfies (\ref{eqNCA_6_1}).
  
  This follows from the fact
        \[
        \begin{split}
      \bar{a}\bar{\gamma}_2+(1-\bar{a})\bar{\gamma}_3-\bar{\gamma}_1 & =\frac{s}{(1-a)\bar{s}+as}\left(a\gamma_2+(1-a)\gamma_3+a\gamma_1\Big(\frac{s}{\bar{s}}-1\Big)\right)- \frac{s}{\bar{s}}\gamma_1\\
      & =\frac{s}{(1-a)\bar{s}+as}(a\gamma_2+(1-a)\gamma_3-\gamma_1). 
        \end{split}
     \]

 (e) $\Lambda$ satisfies (\ref{eqNCA_6_2}) if and only if $\bar{\Lambda}$ satisfies (\ref{eqNCA_6_2}).
  
   This is because of 
\[
  \begin{split}
   &\bar{a}(\bar{\gamma}_2+\bar{\mu})+(1-\bar{a})(\bar{\gamma}_3+\bar{\beta})-(\bar{\gamma}_1+\bar{\alpha})\\
    & =\frac{s}{(1-a)\bar{s}+as}\left(a(\gamma_2+\mu)+(1-a)(\gamma_3+\beta)+a(\gamma_1+\alpha)\Big(\frac{s}{\bar{s}}-1\Big)\right)- \frac{s}{\bar{s}}(\gamma_1+\alpha)\\
   & =\frac{s}{(1-a)\bar{s}+as}(a(\gamma_2+\mu)+(1-a)(\gamma_3+\beta)-(\gamma_1+\alpha)). 
   \end{split}
\]

 (f)  Let $n\ge 2$, then $\Lambda$ satisfies (\ref{eqNCA_6_3}) if and only if $\bar{\Lambda}$ satisfies (\ref{eqNCA_6_3}).

Using the definition of $\bar{s}$ and $\bar{\alpha}$, we have
\begin{equation}\label{eqPre1_V_3}
  \begin{split}
    \frac{1}{\bar{s}}+\frac{\bar{\alpha}}{n-1} 
        &=\frac{s}{\bar{s}}\Big(1+\frac{as}{p'}\Big)^{-1}\Big(1+\frac{as}{p'}\Big)\Big(\frac{1}{s}+\frac{\alpha}{n-1}\Big)\\
        & =\frac{s}{\bar{s}}\Big(1-a+\frac{s}{\bar{s}}a\Big)^{-1}\left(\frac{1}{s}+\frac{\alpha}{n-1}+a\frac{s}{p'}\Big(\frac{1}{s}+\frac{\alpha}{n-1}\Big)\right).
        \end{split}
\end{equation}
By the definition of $\bar{\mu}$ and $\bar{\beta}$, and  the second and third equations in (\ref{eqPre1_5}), we have 
\begin{equation}\label{eqPre1_V_4}
 \begin{split}
& \bar{a}\left(\frac{1}{\bar{p}}+\frac{\bar{\mu}-1}{n-1}\right)+(1-\bar{a})\Big(\frac{1}{\bar{q}}+\frac{\bar{\beta}}{n-1}\Big)\\
& = \frac{s}{\bar{s}}\Big(1-a+\frac{s}{\bar{s}}a\Big)^{-1}\left(a\Big(\frac{1}{p}+\frac{\mu-1}{n-1}\Big)+(1-a)\Big(\frac{1}{q}+\frac{\beta}{n-1}\Big)+a\frac{s}{p'}\Big(\frac{1}{s}+\frac{\alpha}{n-1}\Big)\right). 
\end{split}
\end{equation}
So (f) follows from  (\ref{eqPre1_V_3}) and (\ref{eqPre1_V_4}).

  (g) $1/s\le a/p+(1-a)/q$ if and only if $1/\bar{s}\le \bar{a}+(1-\bar{a})/\bar{q}$, and $1/s\ge a(1/p-1/n)+(1-a)/q$ if and only if $1/\bar{s}\ge \bar{a}(1-1/n)+(1-\bar{a})/\bar{q}$.   
    
   The first part follows from 
 \begin{equation}\label{eqPre2_1}
     \begin{split}
        \bar{a}+ \frac{1-\bar{a}}{\bar{q}}-\frac{1}{\bar{s}} 
         & =\frac{s}{\bar{s}(1-a+as/\bar{s})}\Big(a+\frac{1-a}{q}-\frac{1-a+as/\bar{s}}{s}\Big)\\
         & =\frac{s}{\bar{s}(1-a+as/\bar{s})}\Big(a+\frac{1-a}{q}-\frac{1+as/p'}{s}\Big)\\
        & =\frac{s}{\bar{s}(1-a+as/\bar{s})}\Big(\frac{a}{p}+\frac{1-a}{q}-\frac{1}{s}\Big).
        \end{split}
     \end{equation}
    The second part follows from (\ref{eqPre2_1}) and the definition of $\bar{a}$, through the following computation 
     \[
        \bar{a}\Big(1-\frac{1}{n}\Big)+ \frac{1-\bar{a}}{\bar{q}}-\frac{1}{\bar{s}}  = \bar{a}+ \frac{1-\bar{a}}{\bar{q}}-\frac{1}{\bar{s}}-\frac{\bar{a}}{n}
         =\frac{s}{\bar{s}(1-a+as/\bar{s})}\left(a\Big(\frac{1}{p}-\frac{1}{n}\Big)+\frac{1-a}{q}-\frac{1}{s}\right). 
     \]
   
  (h) $\Lambda$ satisfies (\ref{eqNCA_7}) if and only if $\bar{\Lambda}$ satisfies (\ref{eqNCA_7}), and $\Lambda$ satisfies (\ref{eqNCB_5}) if and only if $\bar{\Lambda}$ satisfies (\ref{eqNCB_5}).

 By the definition of $\bar{\Lambda}$, we have $a=0$ if and only if $\bar{a}=0$, and $a=1$ if and only if $\bar{a}=1$. By (\ref{eqPre1_6}) and using $s/\bar{s}=1+s/p'$, we have 
\[
   \frac{1}{p}+\frac{\gamma_2+\mu-1}{n}=\frac{1}{q}+\frac{\gamma_3+\beta}{n}=\frac{1}{s}+\frac{\gamma_1+\alpha}{n} 
   \]
   if and only if 
   \[
   \frac{1}{\bar{p}}+\frac{\bar{\gamma}_2+\bar{\mu}-1}{n}=\frac{1}{\bar{q}}+\frac{\bar{\gamma}_3+\bar{\beta}}{n}=\frac{1}{\bar{s}}+\frac{\bar{\gamma}_1+\bar{\alpha}}{n}.
\]
By (\ref{eqPre1_V_3}) and (\ref{eqPre1_V_4}), 
\[
  \frac{1}{s}+\frac{\alpha}{n-1}=a\Big(\frac{1}{p}+\frac{\mu-1}{n-1}\Big)+(1-a)\Big(\frac{1}{q}+\frac{\beta}{n-1}\Big) 
  \]
if and only if 
\[
   \frac{1}{\bar{s}}+\frac{\bar{\alpha}}{n-1}=\bar{a}\left(1+\frac{\bar{\mu}-1}{n-1}\right)+(1-\bar{a})\Big(\frac{1}{\bar{q}}+\frac{\bar{\beta}}{n-1}\Big). 
\]
(h) then follows from the above in view of the first part of (g).

Now (i) follows from the fact that $\bar{\Lambda}$ satisfies (\ref{eqNCA_1}), (b)-(d) and (h). (ii) follows from the fact that $\bar{\Lambda}$ satisfies (\ref{eqNCA_1}), (a)-(e) and (h). (iii) follows from (g). 
  \end{proof}
  
    \begin{lem}\label{lemPre2_2}
  Let $n\ge 2$, $s, p, q, a, \gamma_1, \gamma_2, \gamma_3, 
\alpha, \beta$ and $\mu$ satisfy (\ref{eqNCA_1})-(\ref{eqNCA_7}) with $\gamma_1, \gamma_2, \gamma_3=0$, $a>0$, $p=1$, and $1/s-(1-a)/q<1$. Then the parameters $\hat{s}, \hat{p}, \hat{q}, \hat{a}, \hat{\gamma}_1, \hat{\gamma}_2, \hat{\gamma}_3$,   defined by (\ref{eqPre2_2_0}),  satisfy (\ref{eqNCA_1}), (\ref{eqNCB_2})-(\ref{eqNCB_4})  with $n$ replaced by $n-1$, and $1/\hat{s}\le \hat{a}/\hat{p}+(1-\hat{a})/\hat{q}$. 
\end{lem}
   \begin{proof}
   Assume $s, p, q, a, \gamma_1, \gamma_2, \gamma_3, 
   \alpha, \beta, \mu$ satisfy (\ref{eqNCA_1})-(\ref{eqNCA_7}) with $\gamma_1, \gamma_2, \gamma_3=0$. For convenience, denote $\Lambda=(s, p, q, a, \gamma_1, \gamma_2, \gamma_3, 
   \alpha, \beta, \mu)$ and $\widehat{\Lambda}=(\hat{s}, \hat{p}, \hat{q}, \hat{a}, \hat{\gamma}_1, \hat{\gamma}_2, \hat{\gamma}_3)$.

Let $b$ and $\lambda$ be defined by  (\ref{eqPre2_2_b}). By the arguments below (\ref{eqPre2_2_b}), we have $0<\hat{a}<1$ and $0\le \lambda\le 1$. By this and the definition of $\hat{s}, \hat{p}, \hat{q}, \hat{a}$, (\ref{eqNCA_1}) holds for $\widehat{\Lambda}$.
Also, by the definition (\ref{eqPre2_2_0}) of $\widehat{\Lambda}$ and (\ref{eqNCA_2}) for $\Lambda$,  (\ref{eqNCB_2}) holds for $\hat{\Lambda}$ with $n$ replaced by $n-1$.

Next, by the definition of $\widehat{\Lambda}$, $\lambda$ and $b$, we have
\begin{equation}\label{eqPre2_2_3}
    \begin{split}
    & \hat{a}(1+\frac{\hat{\gamma}_2-1}{n-1})+(1-\hat{a})\left(\frac{1}{\hat{q}}+\frac{\hat{\gamma_3}}{n-1} \right)\\
        &=     \hat{a}\Big(1+\frac{\mu-1}{n-1}\Big)+(1-\hat{a})\left(\lambda \Big(1+\frac{\mu}{n-1}\Big)+(1-\lambda)\Big(\frac{1}{q}+\frac{\beta}{n-1}\Big)\right)\\
    &=     ab(1+\frac{\mu-1}{n-1})+(1-ab)\left(\frac{a(1-b)}{1-ab}(1+\frac{\mu}{n-1})+\frac{1-a}{1-ab}(\frac{1}{q}+\frac{\beta}{n-1})\right)\\
    & =a\left(b(1+\frac{\mu-1}{n-1})+(1-b)(1+\frac{\mu}{n-1})\right)+(1-a)(\frac{1}{q}+\frac{\beta}{n-1})\\
    &=a\Big(1+\frac{\mu}{n-1}-\frac{b}{n-1}\Big)+(1-a)\Big(\frac{1}{q}+\frac{\beta}{n-1}\Big)\\
    &=\frac{n}{n-1}\left(a\Big(1+\frac{\mu-1}{n}\Big)+(1-a)\Big(\frac{1}{q}+\frac{\beta}{n}\Big)-\frac{1}{ns}\right), 
      \end{split}
       \end{equation}
   where the definition of $b$  is used in the last step.  
   On the other hand, we have,    by using the definition of $\hat{\Lambda}$ and $\lambda$,  that 
   \[
      \frac{1}{\hat{s}}+\frac{\hat{\gamma}_1}{n-1}= \frac{1}{s}+\frac{\alpha}{n-1}. 
   \]
   Since $\Lambda$ satisfies (\ref{eqNCA_5}), by the above and (\ref{eqPre2_2_3}), we have (\ref{eqNCB_3}) holds for $\widehat{\Lambda}$. 
   
Using the definition of $\widehat{\Lambda}$ and the fact that $\gamma_1=\gamma_2=\gamma_3=0$, we have, by using the definition of $\hat{\Lambda}$ and $\lambda$, that
 \[
      \hat{a}\hat{\gamma}_2 +(1-\hat{a})\hat{\gamma}_3-\hat{\gamma}_1=a\mu +(1-a)\beta-\alpha=a(\mu+\gamma_2) +(1-a)(\beta+\gamma_3)-(\alpha+\gamma_1).  
         \]
     In view of (\ref{eqNCA_6_2}) for $\Lambda$, (\ref{eqNCB_4}) holds for $\widehat{\Lambda}$. 

Finally, by the definition of $\lambda$ and $\hat{a}$, we have

 \[
      \hat{a}+\frac{1-\hat{a}}{\hat{q}}=a+\frac{1-a}{q}.   
   \] 
   In view of (i) and the assumption $p=1$, we have $1/\hat{s}\le \hat{a}/\hat{p}+(1-\hat{a})/\hat{q}$. 
  \end{proof}


\begin{thebibliography}{[MT]}

 \bibitem{BT} M. Badiale and G. Tarantello,  A Sobolev-Hardy inequality with applications to a nonlinear elliptic equation arising in astrophysics,  \emph{Arch. Ration. Mech. Anal.}  163  (2002), 259–293. 



\bibitem{BBM1} J. Bourgain, H. Brezis and P. Mironescu, Another look at Sobolev spaces, \emph{Optimal control and partial differential equations.}  IOS, Amsterdam, (2001),  439–455. 


\bibitem{BBM2} J. Bourgain, H. Brezis and P. Mironescu, Limiting embedding theorems for $W^{s, p}$ when $s\uparrow 1$ and applications, \emph{J. Anal. Math.}  87 (2002), 77–101.

\bibitem{BCG1} H. Bahouri, J.-Y. Chemin and I. Gallagher, In\'{e}galit\'{e}s de Hardy pr\'{e}cis\'{e}es,  (French) [Sharper Hardy inequalities], \emph{C. R. Math. Acad. Sci. Paris.} 341 (2005), 89-92. 

\bibitem{BCG2} H. Bahouri, J.-Y. Chemin and I. Gallagher, Refined Hardy inequalities,  \emph{Ann. Sc Norm. Super. Pisa Cl. Sci.} 5 (2006), 375-391. 

 \bibitem{CKPW} M. Cannone, G. Karch, D. Pilarczyk and G. Wu, Stability of singular solutions to the Navier-Stokes system, 	arXiv:2012.12714 [math.AP], 2020.
 
\bibitem{CR} X. Cabr\'{e} and X. Ros-Oton, Sobolev and isoperimetric inequalities with monomial weights, \emph{Journal of Differential Equations} 255 (2013), 4312-4336. 

\bibitem{CRS}X. Cabr\'{e}, X. Ros-Oton and J. Serra, Sharp isoperimetric inequalities via the ABP method, \emph{Journal of the European Mathematical Society} 18 (2016), 2971-2998. 

\bibitem{CKN2} L. Caffarelli, R. Kohn and L. Nirenberg, Partial regularity of suitable weak solutions of the Navier–Stokes equations, \emph{Comm. Pure Appl. Math.}  35 (1982), 771–831.

\bibitem{CKN} L. Caffarelli, R. Kohn and L. Nirenberg, First order interpolation inequalities with weights, \emph{Composition Math.}  53 (1984),  259--275. 

\bibitem{CW} F. Catrina and Z.Q. Wang, On the Caffarelli-Kohn-Nirenberg inequalities: Sharp constants, existence (and nonexistence), and symmetry of extremal functions, \emph{Comm. Pure Appl. Math.}  54 (2001), 229-258. 


\bibitem{DEL}J. Dolbeault, M.J. Esteban  and M. Loss,  Rigidity versus symmetry breaking via nonlinear flows on cylinders and Euclidean spaces, \emph{Invent. math.}  206 (2016), 397–440.

\bibitem{FS} R. Frank and R. Seiringer, Non-linear ground state representations and sharp Hardy inequalities, \emph{J. Funct. Anal.} 255 (2008), 3407–3430.


\bibitem{GT} D. Gilbarg and N. Trudinger, Elliptic partial differential equations of second
order, Reprint of the 1998 edition, Classics in Mathematics, \emph{Springer-Verlag}, Berlin, 2001.

\bibitem{Karch} G. Karch and D. Pilarczyk, Asymptotic stability of Landau solutions to Navier-Stokes system, \emph{Arch. Ration. Mech. Anal.}  202 (2011),   115-131. 

\bibitem{KPS}G. Karch, D. Pilarczyk and M.E. Schonbek, $L^2$-asymptotic stability of singular solutions to the Navier-Stokes system of equations in $\mathbb{R}^3$, \emph{J. Math. Pures Appl.} 108 (2017), 14-40. 


\bibitem{LLY1} L.  Li,  Y.Y.  Li and X. Yan, Homogeneous solutions of stationary Navier-Stokes equations with isolated singularities on the unit sphere. I. One singularity, \emph{Arch. Ration. Mech. Anal.} 227 (2018), 1091--1163.


\bibitem{LLY}L. Li, Y.Y. Li and X.Yan, Homogeneous solutions of stationary Navier-Stokes equations with isolated singularities on the unit sphere. II. Classification of axisymmetric no-swirl solutions, \emph{Journal of Differential Equations} 264  (2018),   6082-6108.

\bibitem{LY} Y.Y. Li and X. Yan,  Asymptotic  stability  of homogeneous solutions to the Navier-Stokes equations in $\mathbf{R}^3$,  \emph{Journal of Differential Equations} 297 (2021), 226-245. 

\bibitem{LZZ} Y. Y. Li, J. Zhang and T. Zhang, Asymptotic stability of Landau solutions to Navier-Stokes system under $L^p$-perturbations, arXiv:2012.14211 [math.AP], 2020. 

\bibitem{Lin}C.S. Lin, Interpolation inequalities with weights, 
\emph{Comm. Partial Differential Equations} 11 (1986), 1515–1538.

\bibitem{MS} V.G. Maz’ya and T. Shaposhnikova, On the Bourgain, Brezis, and Mironescu theorem concerning limiting embeddings of fractional Sobolev spaces, \emph{J. Funct. Anal.} 195  (2002),   230–238; \emph{Erratum, J. Funct. Anal.} 201  (2003),  298–300.

\bibitem{Mazja} V.G. Maz'ja, Sobolev spaces, Translated from the Russian by T. O. Shaposhnikova. Springer Series in Soviet Mathematics. \emph{Springer-Verlag}, Berlin, 1985.


\bibitem{NS}H. Nguyen and M. Squassina, Fractional Caffarelli–Kohn–Nirenberg inequalities, \emph{J. Funct. Anal.} 274 (2018),  2661–2672. 


\bibitem{NS2} H. Nguyen and M. Squassina, On Hardy and Caffarelli-Kohn-Nirenberg inequalities. \emph{J. Anal. Math.} 139 (2019),  773–797.


\bibitem{ZZ} J. Zhang and T. Zhang, Global well-posedness of perturbed Navier-Stokes system around Landau-solutions, preprint, May 2021. 


\end{thebibliography}
\end{document}